\documentclass[preprint,11pt]{elsarticle}
\usepackage{stmaryrd}
\usepackage{amsmath,amssymb,amsthm}
\usepackage{enumitem}
\usepackage{url}
\usepackage[colorlinks=true,linkcolor=blue,citecolor=blue,urlcolor=blue]{hyperref}

\def \Z{\Bbb Z}
\def \C{\Bbb C}

\def \wt{{\textit wt}}

\def \Span{{\rm Span}}

\def \Res{{\rm Res}}
\def \End{{\rm End}}

\def \Id{{\rm Id}}
\def \Hom{{\rm Hom}}

\def \<{\langle}
\def \>{\rangle}

\def \1{{\bf 1}}

\def \({{\rm (}}
\def \){{\rm )}}

\def \1{{\bf 1}}

\def\Hom{{\rm Hom}}

\def\Res{{\rm Res}}
\def\Span{{\rm Span}}
\def\soc{{\rm soc}}
\newtheorem{Theorem}{Theorem}[section]
\newtheorem{Proposition}[Theorem]{Proposition}
\newtheorem{Lemma}[Theorem]{Lemma}

\newtheorem{Corollary}[Theorem]{Corollary}
\newtheorem{Remark}[Theorem]{Remark}
\newtheorem{Main Theorem}[Theorem]{Main Theorem}

\newtheorem{Definition}[Theorem]{Definition}
\begin{document}

\begin{frontmatter}

\title{A dual theory of Zhu's theory}

\author[addr1]{Hao Wang\footnote{Email address: whaomath@nwu.edu.cn. Supported by a NSFC grant 12001426}}

\address[addr1]{School of Mathematics, Northwest University, Xi'an 710127, Shaanxi, China}

\begin{abstract}
We establish a dual theory of Zhu's $A(\mathcal {V})$-theory, i.e., for a graded vertex operator coalgebra $V$, the coassociative coalgebra $C(V)$ are investigated, any admissible $V$-comodule gives a $C(V)$-comodule, and vice versa. We also prove that $V$ is corational, which means every admissible $V$-comodule is completely reducible, if and only if its dual vertex operator algebra $V'$ is rational.
\end{abstract}

\begin{keyword}
algebras, coalgebras, vertex operator algebras, vertex operator coalgebras, admissible modules, admissible comodules, rationality, corationality.
\MSC[2020] 17B69
\end{keyword}

\end{frontmatter}

\section{Introduction}
The idea of duality is ubiquitous in mathematics. For example, in analysis, when we have a function $y=f(x)$, we also want to know what the inverse function is, since the inverse function may contain certain functional properties that the original function cannot describe explicitly. In general, when we have a map from a set $A$ to another set $B$, if we reverse the orientation of the arrow, we get an inverse map. For a linear space, we are also interested in its dual space, since the dual space also can reflect some properties of the original space.

In classical algebraic theory, associative algebras and coassociative coalgebras are a pair of concepts which is dual to each other. To see this, in the definition of associative algebras obtained from commute diagrams, if we reverse the orientations of all arrows, we get the definition of coassociative coalgebras obtained from commute diagrams. Furthermore, if $A$ is a finite dimensional associative algebra, its dual spase is a coassociative coalgebra, and vice versa. When we put associative algebra structure and coassociative coalgebra structure on a same space and satisfying certain compatible conditions, we get a bialgebra and Hopf algebra, \cite{DNR}.

Lie algebras and Lie coalgebras are another pair of concepts which is dual to each other, and they also have the properties which are similar to the properties between associative algebras and coassociative coalgebras, \cite{M,M1}. There also exists the concept of Lie bialgebras, see \cite{NT,M2}.

Vertex operator algebras is another algebraic structure which describes the symmetric properties of $2$-dimensional conformal field theory, \cite{BPZ,LL}. In certain point of view, it looks like both associative algebras and Lie algebras. Hence, there is a natural question whether there exists the concept of vertex operator coalgebras dual to vertex operator algebras and vice versa. In \cite{H,H1,H2}, the author introduced the concept of vertex operator coalgebras and studied a lot of its properties. Several examples of vertex operator coalgebras also were constructed in those articles. Furthermore, the author also proved that the restricted dual of a vertex operator algebra gives a structure of a vertex operator coalgebra and vice versa. Hence this question has a positive answer.

Now we have vertex operator algebras and vertex operator coalgebras, there is another question whether there exists the concept of vertex operator bialgebras, such that it is both a vertex operator algebra and a vertex operator coalgebra with certain compatible conditions. In \cite{HLX,JKLT,L4,L2,L1}, the authors introduced the concept of vertex bialgebras and studied a lot of its properties. This vertex (operator) bialgebra is a vertex (operator) algebra equipped with a structure of coassociative coalgebra satisfying several compatible conditions. We can see that this kind of vertex (operator) bialgebra is not a vertex operator coalgebra. Hence this question is unsolved, and it is still worth to working on it. We will not solve this question in current paper, and are going to work on it in future.

For a vertex operator algebra $\mathcal {V}$, Zhu constructed an associative algebra $A(\mathcal {V})$ in \cite{Z}. In vertex operator algebras theory, Zhu's $A(\mathcal {V})$-theory plays a crucial role in the study of representation theory, since it builds a bridge between representations of vertex operator algebra $V$ and representations of its Zhu algebra $A(\mathcal {V})$. This theory has a lot of generalizations, see \cite{DLM1,DLM2,DLM3,W}. By Hubbard's theory, we know vertex operator algebras and vertex operator coalgebras are dual to each other, associative algebras and coassociative coalgebras are dual to each other. By Zhu's theory, we can get an associative algebra from a vertex operator algebra. Hence there is a question whether there is a theory dual to Zhu's $A(\mathcal {V})$-theory, i.e., whether there exists a coassociative coalgebra constructed from a vertex operator coalgebra, such that this coassociative coalgebra builds a bridge between representations of vertex operator coalgebra and its associated coassociative coalgebra. This is our motivation.

In Zhu's theory, the associative algebra $A(\mathcal {V})$ is defined to be a quotient space of $\mathcal {V}$. When we take the dual, quotient space becomes a subspace. Hence, for a (graded) vertex operator coalgebra $V$, first we take a subspace $C(V)$ of $V$. then we define a well-defined linear map $\Delta:C(V)\rightarrow C(V)\otimes C(V)$. Taking the restriction of the counit $c$ of $V$ on $C(V)$, still denoted by $c$, we get a triple $(C(V),\Delta,c)$. Then we prove that $\Delta$ is coassociative and $c$ satisfies the left and right counit identities in the definition of coassociative coalgebras. Hence, $(C(V),\Delta,c)$ is a coassociative coalgebra. After getting the coassociative coalgebra $(C(V),\Delta,c)$, we naturally want to study the representations. Let a $\mathbb{N}$-graded vector space $\mathcal {M}$ be an admissible $V$-comodule, we prove that the top level $M_0$ is a $C(V)$-comodule. Hence, we get a functor $\Omega$ from admissible $V$-comodules category to $C(V)$-comodules category. This completes the first half of dual Zhu's theory.

Conversely, we also hope to get an admissible $V$-comodule $\mathcal {L}(M)$ from a given $C(V)$-comodule $M$. This implies $M^*$ is a $C(V)^*$-module. From \cite{H1}, we know the restricted dual $V'$ of $V$ is a vertex operator algebra. By Zhu's theory, there is an associative algebra $A(V')$. First, we prove that there is a well-defined algebraic homomorphism $\Phi:A(V')\rightarrow C(V)^*$. Under certain assumptions, we prove $\Phi$ is an isomorphism. This means $M^*$ is also an $A(V')$-module. Using Zhu's theory, we get an admissible $V'$-module $L(M^*)$. Then the restricted dual $L(M^*)'$ is an admissible $V$-comodule. We define $\mathcal {L}(M)$ to be a sub quotient of $L(M^*)'$. Hence we get a functor $\mathcal {L}$ from $C(V)$-comodules category to admissible $V$-comodules category such that $\Omega\circ\mathcal {L}=\Id$. Furthermore, $\mathcal {L}$ also sends simple objects to simple objects. This completes the other half of dual Zhu's theory.

Using this dual theory, we prove that if $V$ is a corational graded vertex operator coalgebra, i.e., every admissible $V$-comodule is completely reducible, then $C(V)$ is a cosemisimple coassociative coalgebra.

Now we want to study the relations between corationality of a graded vertex operator coalgebra $V$ and rationality of its dual vertex operator algebra $V'$. For this purpose, we need to introduce the higher level dual theory of Zhu's theory. This is completely parallel to $C(V)$-theory. First we define the higher level coassociative algebras $C^k(V)$ for all $k\in\mathbb{N}$, and prove that there is an algebraic homomorphism $\Phi^k:A_k(V')\rightarrow C^k(V)^*,$ where $A_k(V')$ is the higher level Zhu algebra of the dual vertex operator algebra $V'$. Under certain assumptions, we prove $\Phi^k$ are isomorphisms for all $k$. Second we introduce functors $\Omega^k$ and $\mathcal {L}^k$, and prove that $\Omega^k/\Omega^{k-1}\circ\mathcal {L}^k=\Id.$ Then we prove that $V$ is corational if and only if all $C^k(V)$ are cosemisimple. Using a result of \cite{DLM1}, we prove that $V$ is a corational graded vertex operator coalgebra if and only if its dual $V'$ is a rational vertex operator algebra. This is a generalization of a classical result, i.e., $A$ is a semisimple associative algebra if and only if $A^*$ is a cosemisimple coassociative coalgebra.

This paper is organised as follows: In section 2, we recall definitions and properties of associative algebras and coassociative coalgebras, their modules and comodules, and relations between them. We also recall definitions and properties of vertex operator algebras and graded vertex operator coalgebras, admissible modules and admissible comodules, and relations between them. In section 3, we prove that for a given graded vertex operator coalgebra $V$, there is a coassociative coalgebra $C(V)$. Section 4 is the construction of functor $\Omega$. In section 5, we work on the dual relation between $C(V)$ and $A(V')$, we also prove that they are dual to each other under certain assumptions. In section 6, we give the construction of functor $\mathcal {L}$, and prove that $\Omega\circ\mathcal {L}=\Id.$ Furthermore, we prove that $\mathcal {L}(M)$ is not too big, i.e., if $M$ has countable dimension, so are all the homogeneous subspaces of $\mathcal {L}(M)$, and $\mathcal {L}(M)_0=M.$ We also prove that if $V$ is corational, then $C(V)$ is a finite-dimensional cosemisimple coassociative coalgebra. Section 7 is about the higher level coassociative coalgebra $C^k(V)$ associated to a graded vertex operator coalgebra $V$. After introducing the definitions of $C^k(V)$, we first prove that there is a filtration of coassociative coalgebras $C(V)=C^0(V)\subseteq\cdots\subseteq C^k(V)\subseteq\cdots.$
Then following the statements of Section 3-6, we prove that $V$ is corational if and only if $V'$ is rational. Section 8 is a by product, we prove that there is an Lie coalgebra associated to a graded vertex operator coalgebra.

\section{Preliminaries}

In this paper, we work over complex field $\mathbb{C},$ all vector spaces, linear maps will be over $\mathbb{C},$ $\otimes$ means $\otimes _\mathbb{C}$. $\Res_zf(z)$ is the coefficient of $z^{-1}$. Furthermore, we have $$\Res_z\frac{d}{dz}f(z)\cdot g(z)=-\Res_zf(z)\cdot\frac{d}{dz}g(z).$$
For any vector spaces $V,W$, the flipping map $\tau:V\otimes W\rightarrow W\otimes V$ is defined by $$\tau(v\otimes w)=w\otimes v,$$ for any $v\in V,w\in W.$ In this paper, we always view $V\otimes \mathbb{C}=\mathbb{C}\otimes V=V$ naturally. For $n\in\mathbb{C}$, $(z_1+z_2)^n=\sum_{i\geq0}\tbinom{n}{i}z_1^{n-i}z_2^i.$ Let $V$ be a vector space, denote $$V^*=\Hom(V,\mathbb{C})$$ by the dual space of $V$. Let $(\cdot,\cdot)$ be the natural pair between $V^*$ and $V$.

\subsection{Algebras and coalgebras}
In this subsection, we recall several concepts and results about algebras and coalgebras.
\begin{Definition}\cite{DNR}\label{DA2.1}
  An associative algebra is a triple $(A,\mu,\eta)$, where $A$ is a vector space, $\mu: A\otimes A\rightarrow A$ and $\eta:\mathbb{C}\rightarrow A$ are linear maps such that:

  (i) For any $a\in A$, we have $\mu(a\otimes\eta(1))=\mu(\eta(1)\otimes a)=a$. They are called unit identities.

  (ii) For any $a,b,c\in A$, the following identity $$\mu(a\otimes\mu(b\otimes c))=\mu(\mu(a\otimes b)\otimes c)$$ holds. This identity is called associativity. We write $\mu(a\otimes b)=ab$.
\end{Definition}

\begin{Definition}\cite{DNR}\label{DA2.2}
  Let $(A,\mu,\eta)$ be an algebra and $B\subseteq A$ a subspace of $A$. If $(B,\mu,\eta)$ is also an algebra, we say it is a sub algebra of $A$.

  If $\mu(A\otimes B)\subseteq B$, we say $B$ is a left ideal of $A$.

  If $\mu(B\otimes A)\subseteq B$, we say $B$ is a right ideal of $A$.

  $B$ is an ideal if it is both a left ideal and right ideal.

  $A$ is called simple if there is no nontrivial ideal. $A$ is called semisimple if $A$ is a direct sum of simple ideals.
\end{Definition}

\begin{Definition}\cite{DNR}
Let $(A,\mu_A,\eta_A)$ and $(B,\mu_B,\eta_B)$ be two algebras. A linear map $\phi:A\rightarrow B$ is called a homomorphism if $$\phi\circ\eta_A=\eta_B,\phi\circ\mu_A=\mu_B\circ\phi\otimes\phi.$$

  Similarly, we have the definitions of monomorphism, epimorphism and isomorphism, etc.
\end{Definition}

\begin{Definition}\cite{DNR}\label{DAM2.1}
  Let $(A,\mu,\eta)$ be an algebra, $N$ a vector space and $\varrho:A\otimes N\rightarrow N$ a linear map. $(N,\varrho)$ is called a left module of $A$ if

  (i) For any $n\in N$, we have $\varrho(1\otimes n)=n$. This is called left unit identity.

  (ii) For any $a,b\in A,n\in N$, we have $\varrho(a\otimes\varrho(b\otimes n))=\varrho(\mu(a\otimes b)\otimes n)$. This is called associative identity. We write $\varrho(a\otimes n)=an.$
\end{Definition}

\begin{Definition}\cite{DNR}\label{DAM2.2}
  Let $(A,\mu,\eta)$ be an algebra and $(N,\varrho_N)$ a left module of $A$. Let $N^1$ be a sub space of $N$. If $(N^1,\varrho_N|_{N^1})$ is also a left module of $A$, we say $N^1$ is a sub module of $N$.

  $N$ is called simple if there is no nontrivial sub module. $N$ is called semisimple if $N$ is a direct sum of simple sub modules.

  Let $(N^1,\varrho_{N^1})$ and $(N^2,\varrho_{N^2})$ be two left modules of $A$. A linear map $\psi:N^1\rightarrow N^2$ is called an $A$-module homomorphism if $$\psi\circ\varrho_{N^1}=\varrho_{N^2}\circ(\Id\otimes \psi).$$

  Similarly, we have the definitions of monomorphism, epimorphism and isomorphism, etc.
\end{Definition}

\begin{Definition}\cite{DNR}\label{DC2.1}
  A coassociative coalgebra is a triple $(C,\Delta,\epsilon)$, where $C$ is a vector space, $\Delta:C\rightarrow C\otimes C$ and $\epsilon:C\rightarrow \mathbb{C}$ are linear maps such that:

  (i) For any $a\in C$, the following identities $$(\Id_C\otimes \epsilon)\circ \Delta(a)=a=(\epsilon\otimes \Id_C)\circ \Delta(a)$$ hold. They are called counit identities.

  (ii) For any $a\in C$, the following identity $$(\Id_C\otimes \Delta)\circ \Delta(a)=(\Delta\otimes \Id_C)\circ \Delta(a)$$ holds. This identity is called coassociativity.
\end{Definition}

\begin{Remark}
By coassociativity, we will write $\Delta(a)=\sum a'\otimes a''$ for simplicity, and this will cause no confusion.
\end{Remark}

\begin{Definition}\cite{DNR}\label{DC2.2}
  Let $(C,\Delta,\epsilon)$ be a coalgebra and $D\subseteq C$ a subspace of $C$. If $\Delta(D)\subseteq D\otimes D$, we say $(D,\Delta,\epsilon)$ is a sub coalgebra of $C$.

  If $\Delta(D)\subseteq C\otimes D$, we say $(D,\Delta,\epsilon)$ is a left coideal of $C$.

  If $\Delta(D)\subseteq D\otimes C$, we say $(D,\Delta,\epsilon)$ is a right coideal of $C$.

  If $\Delta(D)\subseteq C\otimes D+D\otimes C$ and $\epsilon(D)=0$, we say $(D,\Delta,\epsilon)$ is a coideal of $C$.

  $C$ is called simple if there is no nontrivial sub coalgebra. $C$ is called cosemisimple if $C$ is a direct sum of simple sub coalgebras.
\end{Definition}

\begin{Definition}\cite{DNR}\label{DCM2.1}
  Let $(C,\Delta,\epsilon)$ be a coalgebra, $M$ a vector space and $\Delta_M:M\rightarrow C\otimes M$ a linear map. $(M,\Delta_M)$ is called a left comodule of $C$ if

  (i) For any $m\in M$, the following identity $$(\epsilon\otimes \Id_M)\circ \Delta_M(m)=m$$ holds. This is called left counit identity.

  (ii) For any $m\in M$, the following identity $$(\Id_C\otimes \Delta_M)\circ \Delta_M(m)=(\Delta\otimes \Id_M)\circ \Delta_M(m)$$ holds. This is called coassociative identity.
\end{Definition}

\begin{Definition}\cite{DNR}
Let $(C,\Delta_C,\epsilon_C)$ and $(D,\Delta_D,\epsilon_D)$ be two coalgebras. A linear map $\phi:C\rightarrow D$ is called a homomorphism if $$\epsilon_D\circ\phi=\epsilon_C,\Delta_D\circ\phi=\phi\otimes\phi\circ\Delta_C.$$

  Similarly, we have the definitions of monomorphism, epimorphism and isomorphism, etc.
\end{Definition}

\begin{Remark}
By coassociativity, we will write $\Delta_M(m)=\sum m'\otimes m''$ for simplicity, and this will cause no confusion. In this notation, we have $m'\in C,m''\in M.$
\end{Remark}

\begin{Definition}\cite{DNR}\label{DCM2.2}
  Let $(C,\Delta,\epsilon)$ be a coalgebra and $(M,\Delta_M)$ a left comodule of $C$. Let $M^1$ be a sub space of $M$. If $(M^1,\Delta_M|_{M^1})$ is also a left comodule of $C$, we say $M^1$ is a sub comodule of $M$.

  $M$ is called simple if there is no nontrivial sub comodule. $M$ is called cosemisimple if $M$ is a direct sum of simple sub comodules.

  Let $(M^1,\Delta_{M^1})$ and $(M^2,\Delta_{M^2})$ be two left comodules of $C$. A linear map $\psi:M^1\rightarrow M^2$ is called a $C$-comodule homomorphism if $$(\Id\otimes \psi)\circ \Delta_{M^1}=\Delta_{M^2}\circ \psi.$$

  Similarly, we have the definitions of monomorphism, epimorphism and isomorphism, etc.
\end{Definition}

\begin{Lemma}\cite{DNR}\label{coss}
$C$ is cosemisimple if and only if every $C$-comodule is cosemisimple.
\end{Lemma}

\begin{Proposition}\cite{DNR,K}\label{A-C}
Let $(A,\mu,\eta)$ be an algebra and $(C,\Delta,\epsilon)$ a coalgebra. Let $(M,\Delta_M)$ be a left comodule of $C$. Then, we have

(i) $C^*$ is an algebra with $$(\mu_{C^*}(f\otimes g),a)=\sum f(a')g(a''),\eta_{C^*}(1)=\epsilon,$$ where $f,g\in C^*, a\in C$.

(ii) $A^\circ$ is a coalgebra with $$(\Delta_{A^\circ}(f),a\otimes b)=f(ab),\epsilon_{A^\circ}(f)=f(\eta(1)),$$ where $f\in A^\circ, a,b\in A.$ If $\dim A<\infty,$ then $A^\circ=A^*.$

(iii) $M^*$ is a left $C^*$-module with $(\varrho_{M^*}(f\otimes m^*))m=\sum f(m')m^*(m'')$, where $f\in C^*,m^*\in M^*,m\in M.$
\end{Proposition}

\subsection{Vertex operator algebras and coalgebras}
In this subsection, we recall several concepts and results about vertex operator algebras and coalgebras.
 \begin{Definition}\cite{LL}
  Let $\mathcal {V}=\oplus_{s\in \mathbb{Z}}\mathcal {V}_s$ be a $\mathbb{Z}$-graded vector space with $\mathcal {V}_s=0$, $0\gg s$, $\dim \mathcal {V}_s<\infty,\forall s\in \mathbb{Z}$, and $\textbf{1}\in \mathcal {V}_0,\omega\in \mathcal {V}_2,Y(\cdot,z):\mathcal {V}\otimes \mathcal {V}\rightarrow \End(\mathcal {V})[[z,z^{-1}]],u\otimes v\mapsto Y(u,z)v=\sum_{t\in\mathbb{Z}}u_tvz^{-t-1}$, where $u_t\in \End(\mathcal {V})$. Then, $(\mathcal {V},Y,\textbf{1},\omega)$ is called a vertex operator algebra if the following hold:

  (i) For $ u,v\in \mathcal {V},$ $u_tv=0$, if $t\gg0$.

  (ii) $Y(\textbf{1},z)v=v,\lim_{z\rightarrow0}Y(v,z)\textbf{1}=v$, for $ v\in \mathcal {V}$.

  (iii) Write $Y(\omega,z)=\sum_{t\in\mathbb{Z}}\omega_tz^{-t-1}=\sum_{t\in\mathbb{Z}}L(t)z^{-t-2}$, then
  \begin{eqnarray*}
  &&\mathcal {V}_s=\{v\in \mathcal {V}|L(0)v=sv\},~Y(L(-1)v,z)=\frac{d}{dz}Y(v,z),\\
  &&[L(p),L(q)]=(p-q)L(p+q)+\delta_{p+q,0}\frac{p^3-p}{12}d,
  \end{eqnarray*}
  where $d\in\mathbb{C}$ is called central charge of $\mathcal {V}$. For $v\in \mathcal {V}_s$, $v$ is said to be homogeneous and the weight $\wt v$ of $v$ is defined to be $s$.

  (iv) For $ u,v\in \mathcal {V}$, we have
  \begin{eqnarray*}
  & &z_0^{-1}\delta(\frac{z_1-z_2}{z_0})Y(u,z_1)Y(v,z_2)-z_0^{-1}\delta(\frac{-z_2+z_1}{z_0})Y(v,z_2)Y(u,z_1)\\
  & &\ \ \ \ \ =z_1^{-1}\delta(\frac{z_2+z_0}{z_1})Y(Y(u,z_0)v,z_2).
  \end{eqnarray*}
  \end{Definition}

\begin{Definition}\cite{LL}
  Let $\mathcal {V}$ be a vertex operator algebra. A {\em weak  $\mathcal {V}$-module} $\mathcal {N}$ is a vector space equipped
with a linear map
\begin{align*}
Y_{\mathcal {N}}:\mathcal {V}&\to (\End \mathcal {N})[[z, z^{-1}]],\\
v&\mapsto Y_{\mathcal {N}}(v,z)=\sum_{s\in\Z}v_sz^{-s-1},\,v_s\in \End \mathcal {N},
\end{align*}
satisfying the following conditions: For any $u\in \mathcal {V},\ v\in \mathcal {V},\ w\in \mathcal {N}$ and $t\in \Z$,
\begin{align*}
&\ \ \ \ \ \ \ \ \ \ \ \ \ \ \ \ \ \ \ \ \ \ \ \ \ \ u_tw=0 \text{ for } t\gg0;\\
&\ \ \ \ \ \ \ \ \ \ \ \ \ \ \ \ \ \ \ \ \ \ \ \ \ \ Y_\mathcal {N}(\1, x)=\Id_\mathcal {N};\\
\begin{split}
&x_{0}^{-1}\delta\left(\frac{z_{1}-z_{2}}{z_{0}}\right)Y_{\mathcal {N}}(u,z_{1})Y_\mathcal {N}(v,z_{2})-z_{0}^{-1}\delta\left(
\frac{z_{2}-z_{1}}{-z_{0}}\right)Y_\mathcal {N}(v,z_{2})Y_\mathcal {N}(u,z_{1})\\
&\quad=x_{2}^{-1}\delta\left(\frac{z_{1}-z_{0}}{z_{2}}\right)Y_\mathcal {N}(Y(u,z_{0})v,z_{2}).
\end{split}
\end{align*}

A weak
 $\mathcal {V}$-module  $\mathcal {N}$ is called an \textit{admissible $\mathcal {V}$-module} if $\mathcal {N}$ has a $\Z_{\geq
0}$-gradation $\mathcal {N}=\bigoplus_{s\in\Z_{\geq 0}}N_s$ such
that
\begin{align*}\label{AD1}
v_tN_s\subset N_{\wt{v}+s-t-1}
\end{align*}
for any homogeneous $v\in \mathcal {V}$ and $s,\,t\in\Z$.

An admissible $\mathcal {V}$-module $\mathcal {N}$ is said to be
\textit{irreducible} if $\mathcal {N}$ has no non-trivial admissible
$\mathcal {V}$-submodule. When an admissible $\mathcal {V}$-module $\mathcal {N}$ is
direct sum of irreducible admissible submodules, $\mathcal {N}$ is called
\textit{completely reducible}.

$\mathcal {V}$ is called rational if every admissible $\mathcal {V}$-module is completely reducible.
\end{Definition}

\begin{Definition}\cite{H,H2}\label{Def2.17}
  A vertex coalgebra is a triple $(V,\Yup (z),c)$, where $V$ is a vector space, $c: V\rightarrow\mathbb{C}$ is a linear map, $\Yup (z):V\rightarrow V\otimes V[[z,z^{-1}]]$, $v\mapsto \sum_{k\in\mathbb{Z}}\Delta_k(v)z^{-k-1}$ is a linear map, which satisfy following conditions:

  (i) For any $v\in V$, we have $\Delta_k(v)$ is a finite sum and $\Delta_k(v)=0$ for $k<<0.$

  (ii) For any $v\in V$, we have
    \begin{eqnarray}\label{LU3.1}
      (c\otimes \Id_V)\circ\Yup (z)v=v,
    \end{eqnarray} and
    \begin{eqnarray}\label{CI3.2}
      (\Id_V\otimes c)\circ\Yup (z)v\in V[[z]], \label{CI3.2}\\
      \lim_{z\rightarrow 0}(\Id_V\otimes c)\circ\Yup (z)v=v.\label{CI3.3}
    \end{eqnarray}

  (iii) The following identity
    \begin{eqnarray}\label{JI3.4}
      &&z_0^{-1}\delta(\frac{z_1-z_2}{z_0})(\Id_V\otimes \Yup (z_2))\circ \Yup (z_1)-z_0^{-1}\delta(\frac{z_2-z_1}{-z_0})(\tau\otimes\Id_V)\circ(\Id_V\otimes \Yup (z_1))\circ \Yup (z_2)\nonumber\\
      &&\ \ \ \ \ =z_1^{-1}\delta(\frac{z_2+z_0}{z_1})(\Yup (z_0)\otimes \Id_V)\circ \Yup (z_2)
    \end{eqnarray}
  holds on $V$. This identity is also called Jacobi identity.

  $A$ vertex coalgebra $(V,\Yup(z),c)$ is called graded if there is a $\mathbb{Z}$-gradation $V=\oplus_{s\in\mathbb{Z}}V_s$ on $V$ such that each homogeneous space is finite dimensional and $V_s=0$ for $s<<0.$ View $V\otimes V$ as a $\mathbb{Z}$-graded vector space with natural gradation, i.e., $$(V\otimes V)_t=\oplus_{s\in\mathbb{Z}}V_{t-s}\otimes V_{s}.$$ $\Delta_k$ is a homogeneous linear map of degree $k+1$, i.e., for any $v\in V_s$, we have $$\Delta_k(v)\in (V\otimes V)_{s+k+1}.$$

  A $\mathbb{Z}$-graded vertex coalgebra is called a $\mathbb{Z}$-graded vertex operator coalgebra, if there is another linear map $\rho:V\rightarrow \mathbb{C}$, and write $(\rho\otimes \Id_V)\circ\Yup (z)=\sum_{k\in\mathbb{Z}}L(k)z^{k-2}$, then $L(k)\in\End(V)$ and the following Virasoro identity
    \begin{eqnarray}\label{VI3.5}
     [L(k),L(j)]=(k-j)L(k+j)+\frac{k^3-k}{12}\delta_{k+j,0}d
    \end{eqnarray}
  holds on $V$. $d$ is called the rank of $V$. Furthermore, $V_s$ is an eigenspace of $L(0)$ with eigenvalue $s$ and the following identity
    \begin{eqnarray}\label{L(1D)3.6}
     (L(1)\otimes\Id_V)\circ \Yup (z)=\frac{d}{dz}\Yup (z)
    \end{eqnarray}
  holds on $V$. For any $v\in V_s$, we say $v$ is homogeneous of weight $s$, denoted by $s=\wt v.$

  If there is no confusion, we may say $V$ is a graded vertex operator coalgebra for simplicity.
\end{Definition}

\begin{Remark}\label{Rmk3.2}
  (i) Suppose $v\in V_t$, from definition, we know $\Delta_k(v)\in\oplus_{s\in\mathbb{Z}}V_{t+k+1-s}\otimes V_s$. Hence we can write $$\Delta_k(v)=\sum\sum_{s\in\mathbb{Z}}v'_{t+k+1-s}\otimes v''_{s}$$with $v_i',v_i''\in V_i$, where the second tensor factors are linearly independent. In this way, we have $$\Yup (z)v=\sum_{k\in\mathbb{Z}}\sum\sum_{s\in\mathbb{Z}}v'_{t+k+1-s}\otimes v''_{s}z^{-k-1}.$$

  (ii) From identity (\ref{LU3.1}), we have $$c(v'_{t+k+1-s})=0$$for $k\neq-1,$ and $$v=\sum\sum_{s\in\mathbb{Z}}c(v'_{t-s})\otimes v''_{s}.$$ By choosing one of $v_s''$ to be $v$, we get $s=t$, and $c(v_{t-s}')=0$ if $v_s''\neq v$, and $c(v_{t-s}')=1$ if $v_s''= v.$

  (iii) For a $\mathbb{Z}$-graded vertex operator coalgebra, we know $L(k)$ is a homogeneous linear map with degree $-k.$
\end{Remark}

\begin{Proposition}\cite{H2}
In the definition of vertex coalgebras, Jacobi identity (\ref{JI3.4}) is equivalent to following two conditions:

(i) (Weak cocommutativity) For $v\in V$, there exists $q\in \mathbb{N}$ such that $$(z_1-z_2)^q(\Id_V\otimes \Yup (z_2))\circ \Yup (z_1)v=(z_1-z_2)^q(\tau\otimes\Id_V)\circ(\Id_V\otimes \Yup (z_1))\circ \Yup (z_2)v.$$

(ii) (Weak coassociativity) For $v\in V$, there exists $q\in \mathbb{N}$ such that $$(z_0+z_2)^q(\Yup (z_0)\otimes \Id_V)\circ \Yup (z_2)v=(z_0+z_2)^q(\Id_V\otimes \Yup (z_2))\circ \Yup (z_0+z_2)v.$$
\end{Proposition}

\begin{Lemma}\label{L2.1}\cite{H2}
  Let $(V,\Yup (z),c)$ be a graded vertex coalgebra. Then $$\tau\circ\Yup(z)=\Yup(-z)\circ\exp^{zD^*},$$ where $D^*=\Res_zz^{-2}(\Id_V\otimes c)\circ\Yup(z)$. It is obvious that $D^*$ is homogeneous of degree $-1$. This is called skew-symmetry.

  Furthermore, if $(V,\Yup (z),c,\rho)$ is a graded vertex operator coalgebra, we have$$\tau\circ\Yup(z)=\Yup(-z)\circ\exp^{zL(1)}.$$
\end{Lemma}

\begin{Definition}\cite{H,H2}
   Let $(V,\Yup (z),c,\rho)$ be a graded vertex operator coalgebra. Let $\mathcal {M}$ be a vector space and $\Yup_\mathcal {M}(z):\mathcal {M}\rightarrow V\otimes \mathcal {M}[[z,z^{-1}]]$ a linear map. $(\mathcal {M},\Yup_\mathcal {M}(z))$ is called a weak $V$-comodule, if

  (i) For any $m\in \mathcal {M}$, write $\Yup_\mathcal {M}(z)=\sum_{k\in\mathbb{Z}}\Delta_{\mathcal {M},k}z^{-k-1}$, we have $\Delta_{\mathcal {M},k}(m)$ is a finite sum and $\Delta_{\mathcal {M},k}(m)=0$ for $k<<0.$

  (ii) For any $m\in \mathcal {M}$, we have
    \begin{eqnarray}\label{LUM3.1}
      (c\otimes \Id_\mathcal {M})\circ\Yup_\mathcal {M} (z)m=m.
    \end{eqnarray}

  (iii) The following identity
    \begin{eqnarray}\label{JIM3.4}
      &&z_0^{-1}\delta(\frac{z_1-z_2}{z_0})(\Id_V\otimes \Yup_\mathcal {M} (z_2))\circ \Yup_\mathcal {M} (z_1)\nonumber\\
      &&~~~~~~~~~~-z_0^{-1}\delta(\frac{z_2-z_1}{-z_0})(\tau\otimes\Id_\mathcal {M})\circ(\Id_V\otimes \Yup_\mathcal {M} (z_1))\circ \Yup_\mathcal {M} (z_2)\nonumber\\
      &&=z_1^{-1}\delta(\frac{z_2+z_0}{z_1})(\Yup (z_0)\otimes \Id_\mathcal {M})\circ \Yup_\mathcal {M} (z_2)
    \end{eqnarray}
  holds on $\mathcal {M}$. This identity is also called Jacobi identity.

  (iv) Write $(\rho\otimes \Id_\mathcal {M})\circ\Yup_\mathcal {M} (z)=\sum_{k\in\mathbb{Z}}L(k)z^{k-2}$, then $L(k)\in\End(\mathcal {M})$ and the following Virasoro identity
    \begin{eqnarray}\label{VIM3.5}
     [L(k),L(j)]=(k-j)L(k+j)+\frac{k^3-k}{12}\delta_{k+j,0}d
    \end{eqnarray}
  holds on $\mathcal {M}$.

  (v) The following identities
    \begin{eqnarray}\label{L(1D)M3.6}
     (L(1)\otimes\Id_\mathcal {M})\circ \Yup_\mathcal {M} (z)=\frac{d}{dz}\Yup_\mathcal {M} (z)=\Yup_\mathcal {M}(z)\circ L(1)-(\Id_V\otimes L(1))\circ\Yup_\mathcal {M}(z)
    \end{eqnarray}
  holds on $\mathcal {M}$.

  $(\mathcal {M},\Yup_\mathcal {M}(z))$ is called an admissible $V$-comodule, if there is a $\mathbb{N}$-grading on $\mathcal {M}$, such that $\mathcal {M}=\oplus_{t\in\mathbb{N}}M_t$, and $\Delta_k$ is a homogeneous linear map of degree $k+1$, i.e., for any $m\in M_t$, we have $$\Delta_{\mathcal {M},k}(m)\in (V\otimes \mathcal {M})_{t+k+1}=\oplus_{s\in\mathbb{N}}V_{t+k+1-s}\otimes M_s.$$

  Let $(\mathcal {M},\Yup_\mathcal {M}(z))$ be an admissible $V$-comodule, $\mathcal {M}^1$ a subspace of $\mathcal {M}$. If $(\mathcal {M}^1,\Yup_\mathcal {M}|_{\mathcal {M}^1}(z))$ is also an admissible $V$-comodule, we say $\mathcal {M}^1$ is an admissible sub comodule.

  Let $(\mathcal {M},\Yup_\mathcal {M}(z))$ be an admissible $V$-comodule. $\mathcal {M}$ is called simple if there is no nontrivial admissible sub comodule. $\mathcal {M}$ is called cosemisimple if $\mathcal {M}$ is a direct sum of simple admissible sub comodules.

  Let $(\mathcal {M}^1,\Yup_{\mathcal {M}^1}(z))$ and $(\mathcal {M}^2,\Yup_{\mathcal {M}^2}(z))$ be two admissible $V$-comodules. A linear map $\psi:\mathcal {M}^1\rightarrow \mathcal {M}^2$ is called a $V$-comodule homomorphism if $$(\Id\otimes \psi)\circ \Yup_{\mathcal {M}^1}(z)=\Yup_{\mathcal {M}^2}(z)\circ \psi.$$

  Similarly, we have the definitions of monomorphism, epimorphism and isomorphism, etc.

  A graded vertex operator coalgebra $V$ is called corational if every admissible $V$-comodule is cosemisimple.
\end{Definition}

\begin{Remark}\label{Rmk3.22}
  Let $(V,\Yup (z),c,\rho)$ be a graded vertex operator coalgebra.

  (i) Suppose $\mathcal {M}$ is an admissible $V$-comodule and $m\in M_t$, from definition, we know $\Delta_{\mathcal {M},k}(m)\in\oplus_{s\in\mathbb{N}}V_{t+k+1-s}\otimes M_s$. Hence we can write $$\Delta_{\mathcal {M},k}(m)=\sum\sum_{s\in\mathbb{N}}m'_{t+k+1-s}\otimes m''_{s}$$with $m_i'\in V_i,m''_i\in M_i$, where the second tensor factors are linearly independent. In this way, we have $$\Yup_\mathcal {M} (z)m=\sum_{k\in\mathbb{Z}}\sum\sum_{s\in\mathbb{N}}m'_{t+k+1-s}\otimes m''_{s}z^{-k-1}.$$Hence, identity (\ref{LUM3.1}) implies $$c(m'_{t+k+1-s})=0$$for $k\neq-1,$  $s=n$, and $c(m_{t-s}')=0$ if $m_s''\neq m$, and $c(m_{t-s}')=1$ if $m_s''= m.$

  (ii) $L(k)$ is a homogeneous linear map with degree $-k.$

  (iii) Since $V$ is lower truncated, we can shift the grading of $V$ such that $V$ is $\mathbb{N}$-graded, this makes $V$ to be an admissible $V$-comodule.
\end{Remark}

\begin{Proposition}\cite{H2}\label{WCM}
Let $V$ be a vertex coalgebra. In the definition of weak $V$-comodule, Jocabi identity (\ref{JIM3.4}) is equivalent to following two conditions:

(i) (Weak cocommutativity) For $m\in \mathcal {M}$, there exists $q\in \mathbb{N}$ such that $$(z_1-z_2)^q(\Id_V\otimes \Yup_\mathcal {M} (z_2))\circ \Yup_\mathcal {M} (z_1)m=(z_1-z_2)^q(\tau\otimes\Id_\mathcal {M})\circ(\Id_V\otimes \Yup_\mathcal {M} (z_1))\circ \Yup_\mathcal {M} (z_2)m.$$

(ii) (Weak coassociativity) For $m\in \mathcal {M}$, there exists $q\in \mathbb{N}$ such that $$(z_0+z_2)^q(\Yup (z_0)\otimes \Id_\mathcal {M})\circ \Yup_\mathcal {M} (z_2)m=(z_0+z_2)^q(\Id_V\otimes \Yup_\mathcal {M} (z_2))\circ \Yup_\mathcal {M} (z_0+z_2)m.$$
\end{Proposition}

\begin{Proposition}\cite{H,H1,H2}\label{VC-VA}
Let $(V,\Yup (z),c,\rho)$ be a graded vertex operator coalgebra and $(\mathcal {V},Y,\textbf{1},\omega)$ a vertex operator algebra. Let $(\mathcal {M},\Yup_\mathcal {M}(z))$ be an admissible $V$-comodule and $(\mathcal {N},Y_\mathcal {N})$ an admissible $\mathcal {V}$-module. Suppose $V_0\nsubseteq\ker c,V_2\nsubseteq\ker\rho.$ Then we have

(i) $(\mathcal {V}'=\oplus_{s\in\mathbb{Z}}\mathcal {V}_s^*,\Yup_{\mathcal {V}'}(z),c_{\mathcal {V}'},\rho_{\mathcal {V}'})$ is a graded vertex operator coalgebra with
    \begin{eqnarray*}
     &&c_{\mathcal {V}'}=\textbf{1}^*,\rho_{\mathcal {V}'}=\omega^*,\\
     &&(\Yup_{\mathcal {V}'}(z)f,u\otimes v)=(f,Y(u,z)v),
    \end{eqnarray*}
where $f\in \mathcal {V}',u,v\in\mathcal {V}$, $\textbf{1}^*,\omega^*$ are the dual elements of $\textbf{1},\omega.$

(ii) $(V'=\oplus_{s\in\mathbb{Z}}V_s',Y_{V'}(\cdot,z),\textbf{1}_{V'},\omega_{V'})$ is a vertex operator algebra with
    \begin{eqnarray*}
     &&\textbf{1}_{V'}=c|_{V_0},\omega_{V'}=\rho|_{V_2},\\
     &&(Y(f,z)g,v)=(f\otimes g,\Yup(z)v),
    \end{eqnarray*}
where $f,g\in V',v\in V.$

(iii) $\mathcal {N}'=\oplus_{s\in\mathbb{N}}N_s^*$ is an admissible $\mathcal {V}'$-comodule with $$(\Yup_{\mathcal {N}'}(z)n^*,v\otimes n)=(n^*,Y_\mathcal {N}(v,z)n),$$where $n^*\in \mathcal {N}',v\in\mathcal {V},n\in \mathcal {N}.$

(iv) $\mathcal {M}'=\oplus_{s\in\mathbb{N}}M_s^*$ is an admissible $V'$-module with $$(Y_{\mathcal {M}'}(f,z)m^*,m)=(f\otimes m^*,\Yup_\mathcal {M}(z)m),$$where $f\in V',m^*\in \mathcal {M}',m\in \mathcal {M}.$
\end{Proposition}

\section{Coassociative coalgebra $C(V)$}
Now, we define the coassociative coalgebra $C(V)$ related to a graded vertex operator coalgebra $V$. First, for a linear map $\phi:V\rightarrow V$, let $z^\phi:V\rightarrow V\{z\}$ be a linear map which is defined as $$z^\phi(v)=z^{\lambda}(v),$$ where $v$ is an eigenvector of $\phi$ with eigenvalue $\lambda$, and $V\{z\}=\{\sum_{k\in\mathbb{C}}v_kz^k|v_k\in V\}.$

\begin{Definition}
  Let $V$ be a graded vertex operator coalgebra, define $C(V)$ to be the sub space of $V$ spanned by all elements $v\in V$ which satisfy $$\Res_z(\frac{(1+z)^{L(0)}}{z^2}\otimes \Id_V)\circ\Yup (z)v=0,$$which is equivalent to say$$C(V)=\Span\{v\in V|\Res_z(\frac{(1+z)^{L(0)}}{z^2}\otimes \Id_V)\circ\Yup (z)v=0\}.$$
\end{Definition}

It is obvious that $C(V)$ is nonempty and a vector space. From now on, we always assume $\dim C(V)<\infty.$
\begin{Lemma}\label{CV0}
  Let $V$ be a graded vertex operator coalgebra, for any $v\in C(V)$, define $$\Delta(v)=\Res_z(\frac{(1+z)^{L(0)}}{z}\otimes \Id_V)\circ\Yup (z)v.$$ Then, we have$$(\Id_V\otimes c)\circ \Delta(v)=v=(c\otimes \Id_V)\circ \Delta(v).$$
\end{Lemma}
\proof First, let $v\in V_t$ be homogeneous, we have
    \begin{eqnarray*}
      (\Id_V\otimes c)\circ \Delta(v)&=&(\Id_V\otimes c)\circ\Res_z(\frac{(1+z)^{L(0)}}{z}\otimes \Id_V)\circ\Yup (z)v\\
      &=&\Res_z(\frac{(1+z)^{L(0)}}{z}\otimes \Id_V)\circ(\Id_V\otimes c)\circ\Yup (z)v\\
      &=&v.
    \end{eqnarray*}
Last identity follows from identities (\ref{CI3.2}), (\ref{CI3.3}), and $(1+z)^k=\sum_{i\geq0}\tbinom{k}{i}z^i$ for any $k\in\mathbb{Z}$.

Second, we have
    \begin{eqnarray*}
      &&(c\otimes \Id_V)\circ \Delta(v)\\
      &=&(c\otimes \Id_V)\circ\Res_z\frac{(1+z)^{L(0)}}{z}\otimes \Id_V\circ\Yup (z)v\\
      &=&(c\otimes \Id_V)\circ\Res_z\frac{(1+z)^{L(0)}}{z}\otimes \Id_V(\sum_{k\in\mathbb{Z}}\sum\sum_{s\in\mathbb{Z}}v'_{t+k+1-s}\otimes v''_{s}z^{-k-1})\\
      &=&(c\otimes \Id_V)\circ\Res_z\sum_{k\in\mathbb{Z}}\sum\sum_{s\in\mathbb{Z}}\frac{(1+z)^{t+k+1-s}}{z}v'_{t+k+1-s}\otimes v''_{s}z^{-k-1}\\
      &=&\Res_z\sum_{k\in\mathbb{Z}}\sum\sum_{s\in\mathbb{Z}}\frac{(1+z)^{t+k+1-s}}{z}c(v'_{t+k+1-s})\otimes v''_{s}z^{-k-1}\\
      &=&\Res_z\sum\frac{1}{z}1\otimes v=v.
    \end{eqnarray*}
The fifth identity follows from Remark \ref{Rmk3.2}.

Now for inhomogeneous $v\in V$, using linearity, we can prove the desired identities.              $\hfill\Box$

\begin{Lemma}\label{CV2}
  If $v\in C(V)$, we have $$\Res_z(\frac{(1+z)^{L(0)+k}}{z^{2+l}}\otimes \Id_V)\circ\Yup (z)v=0,$$ for all $l\geq k\geq0.$
\end{Lemma}
\proof Since $$\frac{(1+z)^{L(0)+k}}{z^{2+l}}=\sum_{i\geq0}\tbinom{k}{i}\frac{(1+z)^{L(0)}}{z^{2+l-i}}.$$ By linearity, it is enough to show the identity for $l\geq0.$

Using induction, it is true for $l=0$ by definition. Suppose $\Res_z(\frac{(1+z)^{L(0)}}{z^{2+l}}\otimes \Id_V)\circ\Yup (z)v=0$, now we have
    \begin{eqnarray*}
      0&=&(L(1)\otimes \Id_V)\circ\Res_z(\frac{(1+z)^{L(0)}}{z^{2+l}}\otimes \Id_V)\circ\Yup (z)v\\
      &=&\Res_z(\frac{(1+z)^{L(0)+1}}{z^{2+l}}\otimes \Id_V)\circ (L(1)\otimes \Id_V)\circ\Yup (z)v\\
      &=&\Res_z(\frac{(1+z)^{L(0)+1}}{z^{2+l}}\otimes \Id_V)\circ\frac{d}{dz}\Yup (z)v\\
      &=&-\Res_z\frac{d}{dz}(\frac{(1+z)^{L(0)+1}}{z^{2+l}}\otimes \Id_V)\circ\Yup (z)v\\
      &=&-\Res_z((L(0)+1)\otimes \Id_V)\circ(\frac{(1+z)^{L(0)}}{z^{2+l}}\otimes \Id_V)\circ\Yup (z)v\\
      &&+\Res_z(l+2)(\frac{(1+z)^{L(0)+1}}{z^{2+l+1}}\otimes \Id_V)\circ\Yup (z)v\\
      &=&-\Res_z((L(0)+1)\otimes \Id_V)\circ(\frac{(1+z)^{L(0)}}{z^{2+l}}\otimes \Id_V)\circ\Yup (z)v\\
      &&+\Res_z(l+2)(\frac{(1+z)^{L(0)}}{z^{2+l}}\otimes \Id_V)\circ\Yup (z)v\\
      &&+\Res_z(l+2)(\frac{(1+z)^{L(0)}}{z^{2+l+1}}\otimes \Id_V)\circ\Yup (z)v.
    \end{eqnarray*}
The second identity follows since $L(1)$ is homogeneous of degree $-1$. In last identity, by assumption, the first and second terms are equal to $0$, this implies the last term is also $0.$ This completes the induction.                                     $\hfill\Box$

\begin{Lemma}\label{L(0)delta}
Let $V$ be a graded vertex operator coalgebra. Then, we have $$\Delta_k\circ L(0)=(L(0)\otimes L(0)-k-1)\circ\Delta_k.$$
\end{Lemma}
\proof If $v\in V_t$ is homogeneous, it is trivial. By linearity, it is also true for arbitrary $v\in V.$       $\hfill\Box$

\begin{Lemma}\label{CV3}
Let $V$ be a graded vertex operator coalgebra, and $C(V)$ and $\Delta$ defined as above. If $v\in C(V)$, we have $\Delta(v)\in C(V)\otimes C(V)$ and $$(\Id_V\otimes \Delta)\circ \Delta(v)=(\Delta\otimes \Id_V)\circ \Delta(v).$$
\end{Lemma}
\proof Since $C(V)$ is a subspace of $V$, we have $C(V)\otimes C(V)=V\otimes C(V)\bigcap C(V)\otimes V$.

First, we show $\Delta(v)\in V\otimes C(V)$ for any $v\in C(V)$. To prove this relation, using the definition of $C(V)$, it is enough to show $$(\Id_V\otimes \Res_{z_2}\frac{(1+z_2)^{L(0)}}{z_2^2}\otimes\Id_V)\circ (\Id_V\otimes\Yup (z_2))\circ\Delta(v)=0.$$Now, we have
  \begin{eqnarray*}
    &&(\Id_V\otimes \Res_{z_2}\frac{(1+z_2)^{L(0)}}{z_2^2}\otimes\Id_V)\circ(\Id_V\otimes\Yup (z_2))\circ \Delta(v)\\
    &=&(\Id_V\otimes \Res_{z_2}\frac{(1+z_2)^{L(0)}}{z_2^2}\otimes\Id_V)\circ(\Id_V\otimes\Yup (z_2))\circ \Res_{z_1}\frac{(1+z_1)^{L(0)}}{z_1}\otimes\Id_V\circ\Yup (z_1)v\\
    &=& \Res_{z_1,z_2}(\frac{(1+z_1)^{L(0)}}{z_1}\otimes\frac{(1+z_2)^{L(0)}}{z_2^2}\otimes\Id_V)\circ(\Id_V\otimes\Yup (z_2))\circ\Yup (z_1)v\\
    &=& \Res_{z_0,z_1,z_2}(\frac{(1+z_1)^{L(0)}}{z_1}\otimes\frac{(1+z_2)^{L(0)}}{z_2^2}\otimes\Id_V)\circ\\
    &&\{-z_0^{-1}\delta(\frac{z_2-z_1}{-z_0})(\tau\otimes\Id_V)\circ(\Id_V\otimes \Yup (z_1))\circ \Yup (z_2)v\\
    &&~~~~~~~~~~+z_1^{-1}\delta(\frac{z_2+z_0}{z_1})(\Yup (z_0)\otimes \Id_V)\circ \Yup (z_2)v\}\\
    &=& (\tau\otimes\Id_V)\circ\Res_{z_1,z_2}(\frac{(1+z_2)^{L(0)}}{z_2^2}\otimes\frac{(1+z_1)^{L(0)}}{z_1}\otimes\Id_V)\circ(\Id_V\otimes \Yup (z_1))\circ \Yup (z_2)v\\
    && +\Res_{z_0,z_2}(\frac{(1+z_2+z_0)^{L(0)}}{z_2+z_0}\otimes\frac{(1+z_2)^{L(0)}}{z_2^2}\otimes\Id_V)\circ (\Yup (z_0)\otimes \Id_V)\circ \Yup (z_2)v\\
    &=& (\tau\otimes\Id_V)\circ\Res_{z_1,z_2}(\Id_V\otimes\frac{(1+z_1)^{L(0)}}{z_1}\otimes\Id_V)\circ(\frac{(1+z_2)^{L(0)}}{z_2^2}\otimes \Yup (z_1))\circ \Yup (z_2)v\\
    && +\Res_{z_0,z_2}\sum_{i\geq0,j\geq0,l}\tbinom{L(0)}{i}\tbinom{-1}{j}((1+z_2)^{L(0)-i}\otimes\frac{(1+z_2)^{L(0)}}{z_2^{2+1+j}}\otimes\Id_V)\circ \\
    &&(\Delta_l\otimes \Id_V)\circ \Yup (z_2)vz_0^{-l-1+i+j}\\
    &=& (\tau\otimes\Id_V)\circ\Res_{z_1,z_2}(\Id_V\otimes\frac{(1+z_1)^{L(0)}}{z_1}\otimes\Id_V)\circ(\Id_V\otimes \Yup (z_1))\circ\\
    &&(\frac{(1+z_2)^{L(0)}}{z_2^2}\otimes \Id_V)\circ \Yup (z_2)v\\
    && +\Res_{z_2}\sum_{i\geq0,j\geq0}\tbinom{L(0)}{i}\tbinom{-1}{j}((1+z_2)^{L(0)-i}\otimes\frac{(1+z_2)^{L(0)}}{z_2^{2+1+j}}\otimes\Id_V)\circ (\Delta_{i+j}\otimes \Id_V)\circ \Yup (z_2)v\\
    &=&\Res_{z_2}\sum_{i\geq0,j\geq0}\tbinom{L(0)}{i}\tbinom{-1}{j}(\Delta_{i+j}\otimes \Id_V)\circ(\frac{(1+z_2)^{L(0)+1+j}}{z_2^{2+1+j}}\otimes\Id_V)\circ \Yup (z_2)v\\
    &=&0.
  \end{eqnarray*}
The first identity follows from definition. The second and fourth identities follow from the composition properties of linear maps. The third identity follows from Jacobi identity (\ref{JI3.4}). The fifth and sixth identities follow from the composition properties of linear maps and definitions. For the seventh identity, the first term disappears since $v\in C(V)$. Using Lemma \ref{L(0)delta}, we get the second term. By Lemma \ref{CV2}, we get the eighth identity.

Second, we show $\Delta(v)\in C(V)\otimes V$ for any $v\in C(V)$. To prove this relation, using the definition of $C(V)$, it is enough to show $$(\Res_{z_1}\frac{(1+z_1)^{L(0)}}{z_1^2}\otimes\Id_V \otimes\Id_V)\circ (\Yup (z_1)\otimes\Id_V)\circ\Delta(v)=0.$$Now, we have
  \begin{eqnarray*}
    &&(\Res_{z_1}\frac{(1+z_1)^{L(0)}}{z_1^2}\otimes\Id_V \otimes\Id_V)\circ (\Yup (z_1)\otimes\Id_V)\circ\Delta(v)\\
    &=&\sum_{i\geq0}(\tbinom{L(0)}{i}\otimes\Id_V \otimes\Id_V)\circ (\Delta_{i-2}\otimes\Id_V)\circ\Res_{z_2}(\frac{(1+z_2)^{L(0)}}{z_2}\otimes \Id_V)\circ\Yup (z_2)v\\
    &=&\Res_{z_2}\sum_{i\geq0}(\tbinom{L(0)}{i}\otimes\Id_V \otimes\Id_V)\circ (\frac{(1+z_2)^{L(0)-i}}{z_2}\otimes (1+z_2)^{L(0)+1}\otimes \Id_V)\circ\\
    &&(\Delta_{i-2}\otimes\Id_V)\circ\Yup (z_2)v\\
    &=&\Res_{z_0,z_1,z_2}\sum_{i\geq0}(\tbinom{L(0)}{i}\otimes\Id_V \otimes\Id_V)\circ (\frac{(1+z_2)^{L(0)-i}}{z_2}\otimes (1+z_2)^{L(0)+1}\otimes \Id_V)\circ\\
    &&z_1^{-1}\delta(\frac{z_2+z_0}{z_1})(\Yup (z_0)\otimes\Id_V)\circ\Yup (z_2)vz_0^{i-2}\\
    &=&\Res_{z_0,z_1,z_2}\sum_{i\geq0}z_0^{i-2}(\tbinom{L(0)}{i}\otimes\Id_V \otimes\Id_V)\circ (\frac{(1+z_2)^{L(0)-i}}{z_2}\otimes (1+z_2)^{L(0)+1}\otimes \Id_V)\circ\\
    &&\{z_0^{-1}\delta(\frac{z_1-z_2}{z_0})(\Id_V\otimes \Yup (z_2))\circ \Yup (z_1)v\\
    &&~~~~~~~~~~-z_0^{-1}\delta(\frac{z_2-z_1}{-z_0})(\tau\otimes\Id_V)\circ(\Id_V\otimes \Yup (z_1))\circ \Yup (z_2)v\}\\
    &=&\Res_{z_1,z_2}\sum_{i\geq0}(\tbinom{L(0)}{i}(z_1-z_2)^{i-2}\otimes\Id_V \otimes\Id_V)\circ\\
     &&~~~~~~(\frac{(1+z_2)^{L(0)-i}}{z_2}\otimes (1+z_2)^{L(0)+1}\otimes \Id_V)\circ(\Id_V\otimes \Yup (z_2))\circ \Yup (z_1)v\\
    &&+\Res_{z_1,z_2}\sum_{i\geq0}(\tbinom{L(0)}{i}(-z_2+z_1)^{i-2}\otimes\Id_V \otimes\Id_V)\circ \\
    &&~~~~~~(\frac{(1+z_2)^{L(0)-i}}{z_2}\otimes (1+z_2)^{L(0)+1}\otimes \Id_V)\circ(\tau\otimes\Id_V)\circ(\Id_V\otimes \Yup (z_1))\circ \Yup (z_2)v\\
    &=&\Res_{z_1,z_2}(\frac{(1+z_1)^{L(0)}}{(z_1-z_2)^2}\otimes \frac{(1+z_2)^{L(0)+1}}{z_2}\otimes \Id_V)\circ(\Id_V\otimes \Yup (z_2))\circ \Yup (z_1)v\\
    &&+\Res_{z_1,z_2}(\frac{(1+z_1)^{L(0)}}{(-z_2+z_1)^2}\otimes \frac{(1+z_2)^{L(0)+1}}{z_2}\otimes \Id_V)\circ(\tau\otimes\Id_V)\circ(\Id_V\otimes \Yup (z_1))\circ \Yup (z_2)v\\
    &=&\Res_{z_1,z_2}\sum_{i\geq0}(\tbinom{-2}{i}(-1)^i\frac{(1+z_1)^{L(0)}}{z_1^{2+i}}\otimes \frac{(1+z_2)^{L(0)+1}}{z_2^{1-i}}\otimes \Id_V)\circ(\Id_V\otimes \Yup (z_2))\circ \Yup (z_1)v\\
    &&+\Res_{z_1,z_2}\sum_{i\geq0}(\tbinom{-2}{i}(-1)^{-2-i}\frac{(1+z_1)^{L(0)}}{z_1^{-i}}\otimes \frac{(1+z_2)^{L(0)+1}}{z_2^{3+i}}\otimes \Id_V)\circ\\
    &&~~~~(\tau\otimes\Id_V)\circ(\Id_V\otimes \Yup (z_1))\circ \Yup (z_2)v\\
    &=&\Res_{z_1,z_2}\sum_{i\geq0}(\tbinom{-2}{i}(-1)^i\frac{(1+z_1)^{L(0)}}{z_1^{2+i}}\otimes \frac{(1+z_2)^{L(0)+1}}{z_2^{1-i}}\otimes \Id_V)\circ(\Id_V\otimes \Yup (z_2))\circ \Yup (z_1)v\\
    &&+\Res_{z_1,z_2}(\tau\otimes\Id_V)\circ\sum_{i\geq0}(\tbinom{-2}{i}(-1)^{-2-i}\frac{(1+z_2)^{L(0)+1}}{z_2^{3+i}}\otimes \frac{(1+z_1)^{L(0)}}{z_1^{-i}}\otimes \Id_V)\circ\\
    &&~~~~(\Id_V\otimes \Yup (z_1))\circ \Yup (z_2)v
  \end{eqnarray*}
  \begin{eqnarray*}
    &=&\Res_{z_1,z_2}\sum_{i\geq0}(\tbinom{-2}{i}(-1)^i\Id_V\otimes \frac{(1+z_2)^{L(0)+1}}{z_2^{1-i}}\otimes \Id_V)\circ(\frac{(1+z_1)^{L(0)}}{z_1^{2+i}}\otimes \Id_V\otimes \Id_V)\circ\\
    &&~~~~(\Id_V\otimes \Yup (z_2))\circ \Yup (z_1)v\\
    &&+\Res_{z_1,z_2}(\tau\otimes\Id_V)\circ\sum_{i\geq0}(\tbinom{-2}{i}(-1)^{-2-i}\Id_V\otimes \frac{(1+z_1)^{L(0)}}{z_1^{-i}}\otimes \Id_V)\circ\\
    &&~~~~(\frac{(1+z_2)^{L(0)+1}}{z_2^{3+i}}\otimes\Id_V\otimes\Id_V)\circ(\Id_V\otimes \Yup (z_1))\circ \Yup (z_2)v\\
    &=&\Res_{z_1,z_2}\sum_{i\geq0}(\tbinom{-2}{i}(-1)^i\Id_V\otimes \frac{(1+z_2)^{L(0)+1}}{z_2^{1-i}}\otimes \Id_V)\circ(\frac{(1+z_1)^{L(0)}}{z_1^{2+i}}\otimes \Yup (z_2))\circ \Yup (z_1)v\\
    &&+\Res_{z_1,z_2}(\tau\otimes\Id_V)\circ\sum_{i\geq0}(\tbinom{-2}{i}(-1)^{-2-i}\Id_V\otimes \frac{(1+z_1)^{L(0)}}{z_1^{-i}}\otimes \Id_V)\circ\\
    &&~~~~(\frac{(1+z_2)^{L(0)+1}}{z_2^{3+i}}\otimes \Yup (z_1))\circ \Yup (z_2)v\\
    &=&\Res_{z_1,z_2}\sum_{i\geq0}(\tbinom{-2}{i}(-1)^i\Id_V\otimes \frac{(1+z_2)^{L(0)+1}}{z_2^{1-i}}\otimes \Id_V)\circ(\Id_V\otimes \Yup (z_2))\circ\\
    &&~~~~(\frac{(1+z_1)^{L(0)}}{z_1^{2+i}}\otimes \Id_V)\circ \Yup (z_1)v\\
    &&+\Res_{z_1,z_2}(\tau\otimes\Id_V)\circ\sum_{i\geq0}(\tbinom{-2}{i}(-1)^{-2-i}\Id_V\otimes \frac{(1+z_1)^{L(0)}}{z_1^{-i}}\otimes \Id_V)\circ\\
    &&~~~~(\Id_V\otimes \Yup (z_1))\circ(\frac{(1+z_2)^{L(0)+1}}{z_2^{3+i}}\otimes\Id_V)\circ \Yup (z_2)v\\
    &=&0.
  \end{eqnarray*}
The second identity follows from Lemma \ref{L(0)delta}. The fourth identity follows from Jacobi identity (\ref{JI3.4}). The eighth to eleventh identities follow from composition properties of linear maps. Last identity follows from Lemma \ref{CV2}. Others follow from definitions.

From above two steps, we can see that $\Delta(v)\in C(V)\otimes C(V)$ for any $v\in C(V)$.

The proof of the coassociativity is similar.  Let $v\in C(V)$, we have
  \begin{eqnarray*}
    &&(\Delta\otimes \Id_V)\circ \Delta(v)\\
    &=&(\Res_{z_1}\frac{(1+z_1)^{L(0)}}{z_1}\otimes\Id_V \otimes\Id_V)\circ (\Yup (z_1)\otimes\Id_V)\circ\Delta(v)\\
    &=&\sum_{i\geq0}(\tbinom{L(0)}{i}\otimes\Id_V \otimes\Id_V)\circ (\Delta_{i-1}\otimes\Id_V)\circ\Res_{z_2}(\frac{(1+z_2)^{L(0)}}{z_2}\otimes \Id_V)\circ\Yup (z_2)v\\
    &=&\Res_{z_0,z_1,z_2}\sum_{i\geq0}(\tbinom{L(0)}{i}\otimes\Id_V \otimes\Id_V)\circ (\frac{(1+z_2)^{L(0)-i}}{z_2}\otimes (1+z_2)^{L(0)}\otimes \Id_V)\circ\\
    &&z_1^{-1}\delta(\frac{z_2+z_0}{z_1})(\Yup (z_0)\otimes\Id_V)\circ\Yup (z_2)vz_0^{i-1}
  \end{eqnarray*}
  \begin{eqnarray*}
    &=&\Res_{z_1,z_2}(\frac{(1+z_1)^{L(0)}}{z_1-z_2}\otimes \frac{(1+z_2)^{L(0)}}{z_2}\otimes \Id_V)\circ(\Id_V\otimes \Yup (z_2))\circ \Yup (z_1)v\\
    &&+\Res_{z_1,z_2}(\frac{(1+z_1)^{L(0)}}{-z_2+z_1}\otimes \frac{(1+z_2)^{L(0)}}{z_2}\otimes \Id_V)\circ(\tau\otimes\Id_V)\circ(\Id_V\otimes \Yup (z_1))\circ \Yup (z_2)v\\
    &=&\Res_{z_1,z_2}\sum_{i\geq0}(\tbinom{-1}{i}(-1)^i\frac{(1+z_1)^{L(0)}}{z_1^{1+i}}\otimes \frac{(1+z_2)^{L(0)}}{z_2^{1-i}}\otimes \Id_V)\circ(\Id_V\otimes \Yup (z_2))\circ \Yup (z_1)v\\
    &&+\Res_{z_1,z_2}(\tau\otimes\Id_V)\circ\sum_{i\geq0}(\tbinom{-1}{i}(-1)^{-1-i}\frac{(1+z_2)^{L(0)}}{z_2^{2+i}}\otimes \frac{(1+z_1)^{L(0)}}{z_1^{-i}}\otimes \Id_V)\circ\\
    &&~~~~(\Id_V\otimes \Yup (z_1))\circ \Yup (z_2)v\\
    &=&\Res_{z_1,z_2}\sum_{i\geq0}(\tbinom{-1}{i}(-1)^i\Id_V\otimes \frac{(1+z_2)^{L(0)}}{z_2^{1-i}}\otimes \Id_V)\circ(\Id_V\otimes \Yup (z_2))\circ\\
    &&~~~~(\frac{(1+z_1)^{L(0)}}{z_1^{1+i}}\otimes\Id_V)\circ \Yup (z_1)v\\
    &&+\Res_{z_1,z_2}(\tau\otimes\Id_V)\circ\sum_{i\geq0}(\tbinom{-1}{i}(-1)^{-1-i}\Id_V\otimes \frac{(1+z_1)^{L(0)}}{z_1^{-i}}\otimes \Id_V)\circ\\
    &&~~~~(\Id_V\otimes \Yup (z_1))\circ(\frac{(1+z_2)^{L(0)}}{z_2^{2+i}}\otimes \Id_V)\circ \Yup (z_2)v\\
    &=&\Res_{z_1,z_2}(\Id_V\otimes \frac{(1+z_2)^{L(0)}}{z_2}\otimes \Id_V)\circ(\Id_V\otimes \Yup (z_2))\circ(\frac{(1+z_1)^{L(0)}}{z_1}\otimes \Id_V)\circ \Yup (z_1)v\\
    &=&(\Id_V\otimes \Res_{z_2}\frac{(1+z_2)^{L(0)}}{z_2}\otimes\Id_V)\circ (\Id_V\otimes\Yup (z_2))\circ\Delta(v)\\
    &=&(\Id_V\otimes \Delta)\circ \Delta(v).
  \end{eqnarray*}
The seventh identity follows from Lemma \ref{CV2}. Others are trivial.

This completes the proof.                             $\hfill\Box$

Now, we are ready to state our first theorem.
\begin{Theorem}\label{Thm1}
Let $(V,\Yup (z),c,\rho)$ be a graded vertex operator coalgebra. $C(V)$ and $\Delta$ are defined in Lemma \ref{CV0} and Lemma \ref{CV2}. Still denote $c|_{C(V)}:C(V)\rightarrow \mathbb{C}$ by $c$. If $\dim C(V)<\infty,$ then $(C(V),\Delta,c)$ is a coassociative coalgebra.
\end{Theorem}
\proof By Lemma \ref{CV3}, $\Delta$ is well-defined and satisfies coassociative identity. By Lemma \ref{CV0}, $\Delta$ and $c$ satisfy counit identities.                           $\hfill\Box$

\begin{Proposition}\label{Prop}
For $v\in C(V)$, we have $L(1)v+L(0)v=0.$
\end{Proposition}
\proof First, we know $\Res_z(\Id\otimes c)\circ(\frac{(1+z)^{L(0)}}{z^2}\otimes \Id)\circ\Yup(z)v=0.$ On the other hand,
  \begin{eqnarray*}
  &&\Res_z(\Id\otimes c)\circ(\frac{(1+z)^{L(0)}}{z^2}\otimes \Id)\circ\Yup(z)v\\
  &=&\Res_z\sum_{i\geq0,k\in\mathbb{Z}}\tbinom{L(0)}{i}\otimes c\circ\Delta_k(v)z^{i-2-k-1}\\
  &=&(\Id\otimes c)\circ\Delta_{-2}(v)+(L(0)\otimes c)\circ\Delta_{-1}(v)\\
  &=&(\Id\otimes c)\circ(L(1)\otimes\Id)\circ\Delta_{-1}(v)+(L(0)\otimes c)\circ\Delta_{-1}(v)\\
  &=&(L(1)+L(0))\circ(\Id\otimes c)\circ\Delta_{-1}(v)\\
  &=&L(1)v+L(0)v.
  \end{eqnarray*}
The second and last identity follows from Definition \ref{Def2.17} (ii). The third identity follows from $L(1)$-derivation properties.

Hence we get $L(1)v+L(0)v=0.$                           $\hfill\Box$

\section{$C(V)$-comodules}
Now we construct $C(V)$-comodules from admissible $V$-comodules.

Let $(V,\Yup (z),c,\rho)$ be a graded vertex operator coalgebra and $(\mathcal {M},\Yup_\mathcal {M}(z))$ an admissible $V$-comodule. Let $(C(V),\Delta,c)$ be defined in section 3. Suppose $\dim C(V)<\infty.$
\begin{Lemma}\label{LM4.1}
Let $m\in M_0$. For all $i<-1$, we have $$\Res_z(z^{L(0)+i}\otimes \Id_\mathcal {M})\circ \Yup_\mathcal {M}(z)m=0.$$
\end{Lemma}
\proof By assumption, we have
  \begin{eqnarray*}
    &&\Res_z(z^{L(0)+i}\otimes \Id_\mathcal {M})\circ \Yup_\mathcal {M}(z)m\\
    &=&\Res_z(z^{L(0)+i}\otimes \Id_\mathcal {M})\circ\sum_{k\in\mathbb{Z}}\sum\sum_{s\in\mathbb{N}}m'_{k+1-s}\otimes m''_{s}z^{-k-1}\\
    &=&\Res_z\sum_{k\in\mathbb{Z}}\sum\sum_{s\in\mathbb{N}}m'_{k+1-s}\otimes m''_{s}z^{-k-1+k+1-s+i}\\
    &=&0.
  \end{eqnarray*}
Last identity follows from the conditions of $s$ and $i$.                      $\hfill\Box$

We still denote $\Id_\mathcal {M}$ by the restriction of $\Id_\mathcal {M}$ on $M_0$.
\begin{Theorem}\label{CM0}
For any $m\in M_0$, define $$\Delta_{M_0}(m)=\Res_z(z^{L(0)-1}\otimes \Id_\mathcal {M})\circ \Yup_\mathcal {M}(z)m.$$Then $\Delta_{M_0}(m)\in C(V)\otimes M_0$, and $(M_0,\Delta_{M_0})$ gives a $C(V)$-comodule structure on $M_0.$
\end{Theorem}
\proof This is similar to the proof of Lemma \ref{CV3}.

To show $\Delta_{M_0}(m)\in C(V)\otimes M_0$, it is enough to show $\Delta_{M_0}(m)\in C(V)\otimes \mathcal {M}\bigcap V\otimes M_0.$ Since $m\in M_0$, by definition and Remark \ref{Rmk3.22}, we have
  \begin{eqnarray*}
    \Delta_{M_0}(m)&=&\Res_z(z^{L(0)-1}\otimes \Id_\mathcal {M})\circ \Yup_\mathcal {M}(z)m\\
    &=&\Res_z(z^{L(0)-1}\otimes \Id_\mathcal {M})\circ\sum_{k\in\mathbb{Z}}\sum\sum_{s\in\mathbb{N}}m'_{k+1-s}\otimes m''_{s}z^{-k-1}\\
    &=&\sum_{k\in\mathbb{Z}}\sum m'_{k+1}\otimes m''_{0}\in V\otimes M_0.
  \end{eqnarray*}
Furthermore, we have
  \begin{eqnarray*}
    &&(\Res_{z_1}\frac{(1+z_1)^{L(0)}}{z_1^2}\otimes\Id_V \otimes\Id_\mathcal {M})\circ (\Yup (z_1)\otimes\Id_\mathcal {M})\circ\Delta_{M_0}(m)\\
    &=&\Res_{z_1,z_2}(\sum_{i\geq0}(-1)^i\tbinom{-2}{i}(z_1^{L(0)-2-i}\otimes z_2^{L(0)+i}\otimes \Id_\mathcal {M})\circ(\Id_V\otimes \Yup_\mathcal {M} (z_2))\circ \Yup_\mathcal {M} (z_1)m\\
    &&+\sum_{i\geq0,k,l}\tbinom{-2}{i}(-1)^{-2-i}(\tau\otimes\Id_\mathcal {M})\circ(z_2^{L(0)-2-i}\otimes z_1^{L(0)+i}\otimes \Id_\mathcal {M})\circ\\
    &&~~~~~~(\Id_V\otimes \Yup_\mathcal {M} (z_1))\circ \Yup_\mathcal {M} (z_2)m\\
    &=&\Res_{z_1,z_2}(\sum_{i\geq0}(-1)^i\tbinom{-2}{i}(\Id_V\otimes z_2^{L(0)+i}\otimes \Id_\mathcal {M})\circ(\Id_V\otimes \Yup_\mathcal {M} (z_2))\circ\\
    &&~~~~~~(z_1^{L(0)-2-i}\otimes \Id_\mathcal {M})\circ \Yup_\mathcal {M} (z_1)m\\
    &&+\sum_{i\geq0,k,l}\tbinom{-2}{i}(-1)^{-2-i}(\tau\otimes\Id_\mathcal {M})\circ(\Id_V\otimes z_1^{L(0)+i}\otimes \Id_\mathcal {M})\circ(\Id_V\otimes \Yup_\mathcal {M} (z_1))\circ\\
    &&~~~~~~(z_2^{L(0)-2-i}\otimes\Id_\mathcal {M})\circ \Yup_\mathcal {M} (z_2)m\\
    &=&0.
  \end{eqnarray*}
Following the proof of Lemma \ref{CV3}, we can get the first identity. The second identity follows from the composition properties of linear maps. Last identity follows from Lemma \ref{LM4.1}. This implies $\Delta_{M_0}(m)\in C(V)\otimes \mathcal {M}.$

Now we prove the comodule structure. Let $m\in M_0$, we have
  \begin{eqnarray*}
    (c\otimes \Id_\mathcal {M})\circ\Delta_{M_0}(m)&=&(c\otimes \Id_\mathcal {M})\circ\Res_z(z^{L(0)-1}\otimes \Id_\mathcal {M})\circ \Yup_\mathcal {M}(z)m\\
    &=&\Res_zz^{-1-s}\sum_{k\in\mathbb{Z}}\sum\sum_{s\in\mathbb{N}}c(m'_{k+1-s})\otimes m''_{s}z^{-k-1}\\
    &=&m.
  \end{eqnarray*}
Last identity follows from Remark \ref{Rmk3.22}.

The proof of coassociative identity is similar to the proof of Lemma \ref{CV3}. For any $m\in M_0$, we have
\begin{eqnarray*}
    &&(\Delta\otimes \Id_{M_0})\circ \Delta_{M_0}(m)\\
    &=&(\Res_{z_1}\frac{(1+z_1)^{L(0)}}{z_1}\otimes\Id_V \otimes\Id_\mathcal {M})\circ (\Yup (z_1)\otimes\Id_\mathcal {M})\circ\Res_{z_2}(z_2^{L(0)-1}\otimes \Id_\mathcal {M})\circ \Yup_\mathcal {M}(z_2)m\\
    &=&\Res_{z_1,z_2}\sum_{i\geq0}(\tbinom{-1}{i}(-1)^iz_1^{L(0)-1-i}\otimes z_2^{L(0)-1+i}\otimes \Id_\mathcal {M})\circ(\Id_V\otimes \Yup_\mathcal {M} (z_2))\circ \Yup_\mathcal {M} (z_1)m\\
    &&+\Res_{z_1,z_2}(\tau\otimes\Id_\mathcal {M})\circ\sum_{i\geq0}(\tbinom{-1}{i}(-1)^{-1-i}z_2^{L(0)-2-i}\otimes z_1^{L(0)-i}\otimes \Id_\mathcal {M})\circ\\
    &&~~~~(\Id_V\otimes \Yup_\mathcal {M} (z_1))\circ \Yup_\mathcal {M} (z_2)m\\
    &=&\Res_{z_1,z_2}\sum_{i\geq0}(\tbinom{-1}{i}(-1)^i\Id_V\otimes z_2^{L(0)-1+i}\otimes \Id_\mathcal {M})\circ(\Id_V\otimes \Yup_\mathcal {M} (z_2))\circ\\
    &&~~~~(z_1^{L(0)-1-i}\otimes \Id_\mathcal {M})\circ \Yup_\mathcal {M} (z_1)m\\
    &&+\Res_{z_1,z_2}(\tau\otimes\Id_\mathcal {M})\circ\sum_{i\geq0}(\tbinom{-1}{i}(-1)^{-1-i}\Id_V\otimes z_1^{L(0)-i}\otimes \Id_\mathcal {M})\circ\\
    &&~~~~(\Id_V\otimes \Yup_\mathcal {M} (z_1))\circ(z_2^{L(0)-2-i}\otimes \Id_\mathcal {M})\circ \Yup_\mathcal {M} (z_2)m\\
    &=&\Res_{z_1,z_2}(\Id_V\otimes z_2^{L(0)-1}\otimes \Id_\mathcal {M})\circ(\Id_V\otimes \Yup_\mathcal {M} (z_2))\circ(z_1^{L(0)-1}\otimes \Id_\mathcal {M})\circ \Yup_\mathcal {M} (z_1)m\\
    &=&(\Id_V\otimes \Res_{z_2}z_2^{L(0)-1}\otimes\Id_\mathcal {M})\circ (\Id_V\otimes\Yup_\mathcal {M} (z_2))\circ\Delta_{M_0}(m)\\
    &=&(\Id_V\otimes \Delta_{M_0})\circ \Delta_{M_0}(m).
  \end{eqnarray*}
Following the proof of Lemma \ref{CV3}, we can get the second identity. The third identity follows from the composition properties of linear maps. The fourth identity follows from Lemma \ref{LM4.1}. Others follow from definitions.    \

This completes the proof.                 $\hfill\Box$

\begin{Proposition}\label{P4.3}
Let $(V,\Yup (z),c,\rho)$ be a graded vertex operator coalgebra, $(\mathcal {M}^1,\Yup_{\mathcal {M}^1}(z))$ and $(\mathcal {M}^2,\Yup_{\mathcal {M}^2}(z))$ two admissible $V$-comodules. If $\psi:\mathcal {M}^1\rightarrow \mathcal {M}^2$ is a $V$-comodule homomorphism, then $\psi|_{M^1_0}:M^1_0\rightarrow M^2_0$ is a $C(V)$-comodule homomorphism. Thus we get a functor $\Omega$ from admissible $V$-comodules category to $C(V)$-comodules category defined by $\Omega(\mathcal {M})=M_0$.
\end{Proposition}
\proof By definitions, we have $\Yup_{\mathcal {M}_2}(z)\circ \psi=(\Id_V\otimes \psi)\circ\Yup_{\mathcal {M}_1}(z)$. Hence
  \begin{eqnarray*}
    (\Id_V\otimes \psi)\circ\Delta_{M^1_0}&=&(\Id_V\otimes \psi)\circ\Res_z(z^{L(0)-1}\otimes \Id_{\mathcal {M}_1})\circ \Yup_{\mathcal {M}_1}(z)\\
    &=&\Res_z(z^{L(0)-1}\otimes \Id_{\mathcal {M}_2})\circ(\Id_V\otimes \psi)\circ \Yup_{\mathcal {M}_1}(z)\\
    &=&\Res_z(z^{L(0)-1}\otimes \Id_{\mathcal {M}_2})\circ\Yup_{M^2_0}(z)\circ \psi\\
    &=&\Delta_{M^2_0}\circ \psi.
  \end{eqnarray*}
This completes the proof.                                                       $\hfill\Box$

\section{Dual to Zhu's algebras}
Let $(V,\Yup (z),c,\rho)$ be a graded vertex operator coalgebra, we have a coassociative coalgebra $C(V)$, if $\dim C(V)<\infty$. Hence, we also have an associative algebra $C(V)^*$. From Proposition \ref{VC-VA}, we know $V'$ is a vertex operator algebra. By Zhu's theory, we have an associative algebra $A(V')=V'/O(V')$, where
  \begin{eqnarray*}
   &&O(V')=\Span\{f\circ g=\Res_z\frac{(1+z)^{\wt f}}{z^2}Y(f,z)g| \forall f,g\in V'\},\\
   &&f*g=\Res_z\frac{(1+z)^{\wt f}}{z}Y(f,z)g, \forall f,g\in A(V').
  \end{eqnarray*}
Suppose $\dim A(V')<\infty$, we also have a coassociative coalgebra $A(V')^*.$

\begin{Lemma}\label{Lm5.1}
$v\in C(V)$ if and only if $(f,v)=0$ for all $f\in O(V').$
\end{Lemma}
\proof Let $v\in C(V),f\circ g\in O(V')$. By linearity, we can assume $f$ is homogeneous. Now we have
  \begin{eqnarray*}
   (f\circ g,v)&=&(\Res_z\frac{(1+z)^{\wt f}}{z^2}Y(f,z)g,v)\\
   &=&(\Res_zY(\cdot,z)\circ(\frac{(1+z)^{L(0)}}{z^2}\otimes \Id)(f\otimes g),v)\\
   &=&(f\otimes g,\Res_z(\frac{(1+z)^{L(0)}}{z^2}\otimes \Id)\circ\Yup(z)v)=0.
  \end{eqnarray*}

Conversely, if $v\notin C(V)$, we have $\Res_z(\frac{(1+z)^{L(0)}}{z^2}\otimes \Id)\circ\Yup(z)v\in V\otimes V[[z,z^{-1}]]$ is nonzero. Hence there exist $f,g\in V'$ such that $$(f\otimes g,\Res_z(\frac{(1+z)^{L(0)}}{z^2}\otimes \Id)\circ\Yup(z)v)\neq0.$$ Now we have an element $f\circ g\in O(V')$ such that $(f\circ g,v)\neq0.$                                         $\hfill\Box$

\begin{Proposition}\label{Prop5.2A-C}
(i) Suppose $\dim C(V)<\infty$, there is a well-defined surjective homomorphism $\Phi:A(V')\rightarrow C(V)^*$ of algebras.

(ii) Suppose $\dim A(V')<\infty$, there is a well-defined epimorphism $\Psi:C(V)\rightarrow A(V')^*$ of coalgebras.

(iii) If both $C(V)$ and $A(V')$ are finite dimensional, then $\Phi$ and $\Psi$ are isomorphic.
\end{Proposition}
\proof (i) Since $C(V)$ is a sub space of $V$, we have a linear map $V'\rightarrow V^*\rightarrow C(V)^*,$ denote it by $\phi.$ By Lemma \ref{Lm5.1}, we know $O(V')\subseteq\ker\phi.$ Hence, we get a well-defined linear map $\Phi:A(V')\rightarrow C(V)^*.$

Now $\forall f,g\in A(V'), v\in V$, we have
  \begin{eqnarray*}
   &&(\Phi(f*g),v)=(f*g,\Delta(v))=\sum(f,v')(g,v'')\\
   &=&\sum(\Phi(f),v')(\Phi(g),v'')=(\Phi(f)*\Phi(g),v).
  \end{eqnarray*}
Hence, $\Phi$ is a homomorphism.

If $\dim C(V)<\infty$, we know there is $k\in\mathbb{N}$ such that $C(V)\subseteq\oplus_{i\leq k}V_i$. This implies the linear map $\phi:V'\rightarrow C(V)^*$ is surjective. Hence $\Phi$ is epimorphic.

(ii) Since $A(V')^*=\{v\in (V')^*|(f,v)=0,\forall f\in O(V')\}.$ By Lemma \ref{Lm5.1} and $V\rightarrow (V')^*$, there is a well-defined linear map $\Psi:C(V)\rightarrow A(V')^*.$ Similar to the proof of (i), we can prove the rest statements of (ii).

(iii) Suppose $\dim A(V')<\infty$, then we have a monomorphism $\Psi^*:A(V')^{**}=A(V')\rightarrow C(V)^*$ by (ii). Furthermore, $\Psi^*=\Phi$. Hence, they are isomorphisms.                                                                                   $\hfill\Box$

\begin{Lemma}\label{Lm5.3}
Let $V$ be a graded vertex operator coalgebra. If $\dim A(V')<\infty$, then we have $\dim C(V)<\infty.$
\end{Lemma}
\proof Suppose $\dim A(V')<\infty,$ then there is $k\in\mathbb{N}$, such that $V_i^*\subseteq O(V')$ for all $i>k.$ This means that for any $f\in V_i^*,v\in C(V)$ with $i>k$, we have $(f,v)=0.$ Hence $C(V)\subseteq\oplus_{i\leq k}V_i,$ which implies $\dim C(V)<\infty.$                            $\hfill\Box$

\begin{Theorem}
Let $V$ be a graded vertex operator coalgebra. If $\dim A(V')<\infty$, then $\Phi$ and $\Psi$ are isomorphisms.
\end{Theorem}
\proof This follows from Proposition \ref{Prop5.2A-C} and Lemma \ref{Lm5.3}.                    $\hfill\Box$

\section{Admissible $V$-comodules from $C(V)$-comodules}
Let $(V,\Yup (z),c,\rho)$ be a graded vertex operator coalgebra, we have a coassociative coalgebra $C(V)$ if $\dim C(V)<\infty$. Let $(M,\Delta_M)$ be a left $C(V)$-comodule. Then, $M^*$ is a left $C(V)^*$-module. By Proposition \ref{Prop5.2A-C}, it is also an $A(V')$-module. By Zhu's theory, we get an admissible $V'$-module $L(M^*)$ satisfying $L(M^*)_0=M^*.$ We give the construction of $L(M^*)$ briefly.

Consider $\widehat{V'}=V'\otimes \C[t,t^{-1}]/D(V'\otimes \C[t,t^{-1}])$, where $D=L(-1)\otimes 1+1\otimes\frac{d}{dt}$. Then $\widehat{V'}$ is a $\Z$-graded Lie algebra with triangular decomposition $$\widehat{V'}=\widehat{V'}(-)\oplus\widehat{V'}(0)\oplus\widehat{V'}(+).$$ Furthermore, there is a well-defined Lie algebra epimorphism $\phi:\widehat{V'}(0)\rightarrow A(V')$. So $M^*$ is a $\widehat{V'}(0)$-module. Regard $M^*$ as trivial $\widehat{V'}(-)$-module. Set $$U(M^*)=\mathcal {U}(\widehat{V'})\otimes_{\mathcal {U}(\widehat{V'}(0)+\widehat{V'}(-))}M^*\cong\mathcal {U}(\widehat{V'}(+))\otimes M^*,$$ and $L(M^*)=U(M^*)/J$, where $$J=\{m^*\in U(M^*)|<m^{**},fm^*>=0,\forall m^{**}\in M^{**},f\in\mathcal {U}(\widehat{V'})\}.$$
Then $(L(M^*),Y_{M^*})$ is an admissible $V'$-module with $Y_{M^*}(f,z)=\sum_kf\otimes t^kz^{-k-1}$, for any $f\in V'$.

From above discussion, we know $L(M^*)'$ is a left $V''=V$-comodule. On the other hand, there is an injection $M\rightarrow M^{**}=L(M^*)_0^*\rightarrow L(M^*)'$. If $\dim M<\infty,$ the first injection is also a bijection. Under above injection, we can regard $M$ as a subspace of $L(M^*)'.$ Define $\mathcal {W}(M)$ to be the sub comodule of $L(M^*)'$ generated by $M$, i.e., the smallest sub comodule containing $M$. Define $\mathcal {L}(M)$ to be the quotient $\mathcal {W}(M)/\mathcal {J}$, where $\mathcal {J}$ is the maximal sub comodule of $\mathcal {W}(M)$ which intersects $M$ trivially.

\begin{Lemma}\cite{Z}
Let $\mathcal {N}$ be an admissible $V'$-module, then $N_0$ is an $A(V')$-module with module structure defined as $\varrho_{N_0}(f\otimes n)=f\otimes t^{\wt f-1}n,$ where $n\in N_0,f\in V'$ is homogeneous and extend to all $f\in V'$ by linearity.
\end{Lemma}

\begin{Lemma}\label{Lm6.1}
Suppose $\dim M$ is countable. Then, we have $\dim\mathcal {L}(M)$ is countable. Furthermore, $\mathcal {L}(M)_0=M.$
\end{Lemma}
\proof Define $F^0(M)=M$, suppose we have already defined $F^k(M)$, define $F^{k+1}(M)$ to be the space spanned by the second tensor factors of all $$\Delta_i(m),i\in\mathbb{Z},m\in F^k(M).$$
By Remark \ref{Rmk3.22}, we can see that $F^k(M)\subseteq F^{k+1}(M)$ for all $k$. Hence, we get the following filtration $$F^0(M)\subseteq F^1(M)\subseteq\cdots\subseteq F^k(M)\subseteq\cdots.$$
Since $\Delta_i(m)$ is a finite sum, if $\dim M$ is countable, using induction, we can see that $\dim F^k(M)$ is countable for all $k.$

Now set $F=\cup_kF^k(M)$, we have $\dim F$ is countable. Since $L(M^*)'$ is an admissible $V$-comodule, it is obvious that $F$ is also an admissible $V$-comodule, such that $M\subseteq F\subseteq\mathcal {W}(M).$ Hence, $\mathcal {W}(M)=F$ and $\dim\mathcal {W}(M)$ is countable. Now $\mathcal {L}(M)$ is a quotient comodule of $\mathcal {W}(M)$, it is obvious that $\dim\mathcal {L}(M)$ is countable.

To show $\dim\mathcal {L}(M)_0=M,$ it is enough to show that $\mathcal {W}(M)_0=M$, which is equivalently to show for any $m\in\mathcal {W}(M)$, when the second tensor factors of $\Delta_i(m)$ belong to $\mathcal {W}(M)_0$ for any $i\in\mathbb{Z}$, they belong to $M$. Using above filtration, we prove that $\forall m\in F^k(M)$, when the second tensor factors of $\Delta_i(m)$ belong to $\mathcal {W}(M)_0$, they belong to $M$. By linearity, we can assume $m$ is homogeneous.

Induction on $k$. If $k=0$, for any $m\in M$, then $\Delta_i(m)=\sum m'\otimes m''\in (V\otimes\mathcal {W}(M))_{i+1}$. If $m''\in\mathcal {W}(M)_0$, then $m'\in V_{i+1}$. Hence for any $f\in V_{i+1}^*,m^*\in M^*$, we have
  \begin{eqnarray*}
   &&(\Delta_i(m),f\otimes m^*)=(m,(f\otimes t^i)m^*)\\
   &=&(m,\varrho_{M^*}(f\otimes m^*))=(\Delta_M(m),f\otimes m^*).
  \end{eqnarray*}
Since $(M,\Delta_M)$ is a $C(V)$-comodule, if $m'\in V_{i+1}$, then $m''\in M.$

Suppose it is true for $m\in F^k(M)$. Now let $m\in F^{k+1}(M).$ By definition, there are $m_s'\in V,m_s''\in \mathcal {W}(M)$ with $m_0''=m$ such that $\sum m_s'\otimes m_s''=\Delta_p(x)$ for some $p\in\mathbb{Z}, x\in F^k(M)$. Suppose $m_s'$ are linear independent, let $f\in V'$ be the dual element of $m_0'$, then we have
  \begin{eqnarray*}
   \Delta_i(m)&=&\Delta_i(\sum f(m_s')m_s'')=(\Id\otimes\Delta_i)\circ(f\otimes \Id)\circ\Delta_p(x)\\
   &=&(f\otimes \Id\otimes\Id)\circ(\Id\otimes \Delta_i)\circ\Delta_p(x).
  \end{eqnarray*}
Since $\mathcal {W}(M)$ is an admissible $V$-comodule, by Proposition \ref{WCM}, we have the following weak coassociativity
$$(z_0+z_2)^q(\Yup (z_0)\otimes \Id_{\mathcal {W}(\mathcal {M})})\circ \Yup_{\mathcal {W}(M)} (z_2)x=(z_0+z_2)^q(\Id_V\otimes \Yup_{\mathcal {W}(M)} (z_2))\circ \Yup_{\mathcal {W}(M)} (z_0+z_2)x$$for some $q\in\mathbb{N}$. Hence we can see that $(\Id\otimes \Delta_i)\circ\Delta_p(x)$ is a linear combination of $(\Delta_j\otimes \Id)\circ\Delta_l(x)$. By induction, we can see that if the second tensor factors of $\Delta_l(x)$ belong to $\mathcal {W}(M)_0$, they belong to $M$. Hence the statement is also true for $\Delta_i(m)$. Thus, we get $\mathcal {W}(M)_0=M.$ This means $\mathcal {L}(M)_0=M$ by definition.

This completes the proof.                                                            $\hfill\Box$

\begin{Proposition}
$\mathcal {L}$ is a functor from $C(V)$-comodules category to admissible $V$-comodules category such that $\Omega\circ\mathcal {L}=\Id.$ Furthermore, $\mathcal {L}$ sends simple objects to simple objects.
\end{Proposition}
\proof Let $\psi:M^1\rightarrow M^2$ be a $C(V)$-comodule homomorphism, then $\psi^*:(M^2)^*\rightarrow (M^1)^*$ is a $C(V)^*$-module homomorphism. By Proposition \ref{Prop5.2A-C}, $\psi^*$ is also an $A(V')$-module homomorphism. By Zhu's theory, $\psi^*$ induces a $V'$-module homomorphism from $L((M^2)^*)$ to $L((M^1)^*)$, still denoted by $\psi^*.$ Now, $\psi^*$ induces a $V$-comodule homomorphism $\psi^{**}:L((M^1)^*)'\rightarrow L((M^2)^*)'$ and $\psi^{**}|_{M^1_0}=\psi.$ Define $\mathcal {L}(\psi)=\psi^{**}|_{\mathcal {L}(M^1)}$, we get $\mathcal {L}(\psi):\mathcal {L}(M^1)\rightarrow\mathcal {L}(M^2)$ is a $V$-comodule homomorphism. This proves $\mathcal {L}$ is a functor.

$\Omega\circ\mathcal {L}=\Id$ is obvious by definitions.

Now let $M$ be a simple $C(V)$-comodule. If $\mathcal {L}(M)$ is not simple, then it has a sub comodule $\mathcal {M}=\oplus_{i\geq0}M_i$ such that $M_0\neq0$, Now $M_0$ is a sub $C(V)$-comodule of $M$. Hence $M_0=M$. This implies $\mathcal {M}=\mathcal {L}(M)$ and $\mathcal {L}(M)$ is simple.                                               $\hfill\Box$

\begin{Theorem}
If $V$ is a corational graded vertex operator coalgebra such that $\dim C(V)<\infty$, then $C(V)$ is a cosemisimple coassociative coalgebra.
\end{Theorem}
\proof Let $M$ be a $C(V)$-comodule, then there is an admissible $V$-comodule $\mathcal {L}(M)$ such that $\mathcal {L}(M)_0=M$. If $V$ is corational, then $\mathcal {L}(M)$ is cosemisimple. Hence $\mathcal {L}(M)_0$ is a cosemisimple $C(V)$-comodule, i.e., $M$ is a cosemisimple $C(V)$-comodule. This implies $C(V)$ is cosemisimple.                                 $\hfill\Box$

\begin{Corollary}
If $V$ is a corational graded vertex operator coalgebra such that $\dim C(V)<\infty$, then $V$ has only finitely many inequivalent irreducible admissible comodules.
\end{Corollary}

\section{Higher level Zhu's coalgebras}
In this section section, we introduce a series of coassociative coalgebra $C^k(V)$ associated to a graded vertex operator coalgebra $(V,\Yup (z),c,\rho)$, for every $k\in\mathbb{N} $.

First, define $$C^k(V)=\{v\in V| \Res_z(\frac{(1+z)^{L(0)+k}}{z^{2k+2}}\otimes\Id)\circ\Yup(z)v=0,L(1)v+L(0)v=0\}.$$For any $v\in C^k(V),$ define $$\Delta^k(v)=\sum_{l=0}^k(-1)^l\tbinom{k+l}{l}\Res_z(\frac{(1+z)^{L(0)+k}}{z^{k+l+1}}\otimes\Id)\circ\Yup(z)v.$$

\begin{Lemma}
Suppose $\dim C^k(V)<\infty$. Let $v\in C^k(V).$ Suppose $i\geq j\geq0$, we have $$\Res_z(\frac{(1+z)^{L(0)+k+j}}{z^{2k+2+i}}\otimes\Id)\circ\Yup(z)v=0.$$In particular, $C^k(V)\subseteq C^{k+1}(V)$ for all $k\in\mathbb{N}$.
\end{Lemma}
\proof Using Proposition \ref{Prop}, this is similar to the proof of Lemma \ref{CV2}, we omit it.
  $\hfill\Box$

\begin{Proposition}
Suppose $\dim C^k(V)<\infty$ for all $k\in\mathbb{N}$.

(i) $(C^k(V),\Delta^k,c)$ is a coassociative coalgebra. When $k=0$, we arrive at $C(V).$

(ii) There is a filtration of coassociative coalgebras $$C^0(V)\subseteq C^1(V)\subseteq\cdots\subseteq C^k(V)\subseteq\cdots.$$
\end{Proposition}
\proof (i) This is similar to section 3, we omit it.

(ii) Let $v\in C^k(V)\subseteq C^{k+1}(V)$, by above Lemma, we have
  \begin{eqnarray*}
   &&\Delta^{k+1}(v)=\sum_{l=0}^{k+1}(-1)^l\tbinom{k+1+l}{l}\Res_z(\frac{(1+z)^{L(0)+k+1}}{z^{k+1+l+1}}\otimes\Id)\circ\Yup(z)v\\
   &=&\sum_{l=0}^{k+1}(-1)^l\tbinom{k+1+l}{l}\Res_z(\frac{(1+z)^{L(0)+k}}{z^{k+1+l+1}}+\frac{(1+z)^{L(0)+k}}{z^{k+1+l}}\otimes\Id)\circ\Yup(z)v\\
   &=&\sum_{l=1}^{k+1}(-1)^{l-1}\tbinom{k+l}{l-1}\Res_z(\frac{(1+z)^{L(0)+k}}{z^{k+l+1}}\otimes\Id)\circ\Yup(z)v\\
   &&~~~~~~+(-1)^{k+1}\tbinom{2k+2}{k+1}\Res_z(\frac{(1+z)^{L(0)+k}}{z^{2k+3}}\otimes\Id)\circ\Yup(z)v\\
   &&~~~~\sum_{l=0}^{k+1}(-1)^l\tbinom{k+1+l}{l}\Res_z(\frac{(1+z)^{L(0)+k}}{z^{k+1+l}}\otimes\Id)\circ\Yup(z)v\\
   &=&\sum_{l=1}^{k+1}(-1)^{l}(\tbinom{k+1+l}{l}-\tbinom{k+l}{l-1})\Res_z(\frac{(1+z)^{L(0)+k}}{z^{k+l+1}}\otimes\Id)\circ\Yup(z)v\\
   &&+\Res_z(\frac{(1+z)^{L(0)+k}}{z^{k+1}}\otimes\Id)\circ\Yup(z)v+(-1)^{k+1}\tbinom{2k+2}{k+1}\Res_z(\frac{(1+z)^{L(0)+k}}{z^{2k+3}}\otimes\Id)\circ\Yup(z)v\\
   &=&\sum_{l=0}^{k+1}(-1)^{l}\tbinom{k+l}{l}\Res_z(\frac{(1+z)^{L(0)+k}}{z^{k+l+1}}\otimes\Id)\circ\Yup(z)v\\
   &=&\sum_{l=0}^{k}(-1)^{l}\tbinom{k+l}{l}\Res_z(\frac{(1+z)^{L(0)+k}}{z^{k+l+1}}\otimes\Id)\circ\Yup(z)v=\Delta^k(v).
  \end{eqnarray*}
Again by above Lemma, $C^k(V)$ is a sub coalgebra of $C^{k+1}(V)$.                   $\hfill\Box$

\begin{Proposition}
Suppose $\dim C^k(V)<\infty$.
Let $(\mathcal {M},\Yup_\mathcal {M}(z))$ be an admissible $V$-comodule. Then $\Omega^k(\mathcal {M})=\oplus_{i=0}^kM_i$ is a $C^k(V)$-comodule with comodule structure defined as $$\Delta_{\Omega^k(\mathcal {M})}m=\Res_z(z^{L(0)-1}\otimes \Id)\circ\Yup_\mathcal {M}(z)m,$$where $m\in\Omega^k(\mathcal {M}).$ Furthermore, $\Omega^k$ is a functor from admissible $V$-comodules category to $C^k(V)$-comodules category.
\end{Proposition}
\proof  This is similar to section 4, we omit it.                                       $\hfill\Box$

\begin{Proposition}\label{Prop7.4}
Let $V$ be a graded vertex operator coalgebra. Suppose $\dim C^k(V)<\infty$.

(i) There is a well-defined surjective homomorphism $\Phi^k:A_k(V')\rightarrow C^k(V)^*$ of algebras, where $A_k(V')$ is the higher level Zhu algebra of vertex operator algebra $V'$.

(ii) Suppose $\dim A_k(V')<\infty$, there is a well-defined epimorphism $\Psi^k:C^k(V)\rightarrow A_k(V')^*$ of coalgebras.

(iii) If $A_k(V')$ is finite-dimensional, we have $\Phi^k,\Psi^k$ are isomorphisms.
\end{Proposition}
\proof  This is similar to section 5, we omit it.                                       $\hfill\Box$

\begin{Proposition}
Suppose $\dim C^k(V)<\infty$ for all $k\in\mathbb{N}$.
Let $M$ be a $C^k(V)$-comodule which is not a $C^{k-1}(V)$-comodule. Then there is an admissible $V$-comodule $\mathcal {L}^k(M)$ such that $\Omega^k/\Omega^{k-1}\circ\mathcal {L}^k(M)=M.$

Furthermore, $\mathcal {L}^k$ is a functor from the category of $C^k(V)$-comodules which are not $C^{k-1}(V)$-comodules to admissible $V$-comodules category such that $\Omega^k/\Omega^{k-1}\circ\mathcal {L}^k=\Id$, it also sends simple objects to simple objects.
\end{Proposition}
\proof  This is similar to section 6, we omit it.                                       $\hfill\Box$

\begin{Theorem}\label{7.5}
If $V$ is a corational graded vertex operator coalgebra such that $\dim C^k(V)<\infty$, then $C^k(V)$ is a cosemisimple coassociative coalgebra.
\end{Theorem}
\proof  This is similar to section 6, we omit it.                                       $\hfill\Box$

\begin{Lemma}\cite{DLM1}\label{7.6}
 $\mathcal {V}$ is a rational vertex operator algebra if and only if all its Zhu algebras $A_k(\mathcal {V})$ are finite-dimensional semisimple associative algebras.
\end{Lemma}

\begin{Lemma}\cite{DNR}\label{7.7}
$C$ is a cosemisimple coassociative coalgebra if and only if its dual algebra $C^*$ is a semisimple associative algebra.
\end{Lemma}

\begin{Theorem}
Let $V$ be a graded vertex operator coalgebra such that $\dim A_k(V')<\infty$ for all $k\in\mathbb{N}$. Then $V$ is corational if and only if $V'$ is a rational vertex operator algebra.
\end{Theorem}
\proof By Proposition \ref{Prop7.4}, Theorem \ref{7.5}, Lemma \ref{7.6} and \ref{7.7}, it is obvious that if $V$ is corational, then $V'$ is rational.

Conversely, it is enough to show if all $C^k(V)$ are cosemisimple, then $V$ is corational. For any admissible $V$-comodule $\mathcal {M}=\oplus_{i\in\mathbb{N}}M_i$, let $\soc(\mathcal {M})$ be the sum of all simple sub comodules of $\mathcal {M}$. If $\soc(\mathcal {M})=\mathcal {M}$, we are done.

If not, there is $k\in\mathbb{N}$ such that $\soc(\mathcal {M})_k\neq M_k$. Suppose $k$ is the smallest positive integer such that $\soc(\mathcal {M})_k\neq M_k$. Since $C^k(V)$ is cosemisimple, we have $M_k$ is a cosemisimple $C^k(V)$-comodule, and $\soc(\mathcal {M})_k$ is a proper sub comodule of $M_k$. Thus there is at least one simple sub comodule $X_k$ of $M_k$ intersects $\soc(\mathcal {M})_k$ trivially. Suppose the irreducible admissible $V$-comodule corresponding $X_k$ is $\mathcal {Z}$.

Let $\mathcal {X}$ be the admissible sub $V$-comodule of $\mathcal {M}$ generated by $X_k$. If $\mathcal {X}$ is a simple admissible $V$-comodule, then $\mathcal {X}\subseteq\soc(\mathcal {M})$, this is impossible. Hence $\mathcal {X}$ is reducible. Now $\mathcal {Z}$ is the unique irreducible quotient of $\mathcal {X}$. If there is $l$, such that $X_l\neq Z_l$. Since $X_l$ is a cosemisimple $C^l(V)$-comodule, there is a $C^l(V)$-comodule $S_l$ such that $X_l=Z_l\oplus S_l$. Let $\mathcal {S}$ be the admissible sub $V$-comodule of $\mathcal {X}$ generated by $S_l$. Then we know $\mathcal {S}_k=0.$ Let $\mathcal {T}$ be an irreducible quotient of $\mathcal {S}$, then it is also an irreducible quotient of $\mathcal {X}$, which is different from $\mathcal {Z}$. This is a contradiction. Hence $\mathcal {X}=\mathcal {Z}$, and $\mathcal {X}$ is irreducible. This is also a contradiction, and this contradiction means $\soc(\mathcal {M})=\mathcal {M}$, i.e., $\mathcal {M}$ is completely reducible. Thus $V$ is corational.

This completes the proof.                                $\hfill\Box$

\begin{Remark}
This is a generalization of a classical result, i.e., $A$ is a semisimple associative algebra if and only if its dual coassociative coalgebra $A^*$ is cosemisimple.
\end{Remark}

\section{Relationship with Lie coalgebras}
In this section, we give a relation between Lie coalgebras and vertex operator coalgebras.

Let $U,V,W$ be three vector spaces, for any vector $u\otimes v\otimes w\in U\otimes V\otimes W$, define $\xi(u\otimes v\otimes w)=v\otimes w\otimes u.$ It is trivial that $$\xi=(\Id\otimes \tau)\circ(\tau\otimes \Id),~\xi^2=(\tau\otimes \Id)\circ(\Id\otimes \tau).$$

\begin{Definition}\label{LC5.1}\cite{M}
An Lie coalgebra is a vector space $L$ with a linear map $\Delta:L\rightarrow L\otimes L$, such that:

(i) $\tau\otimes\Delta=-\Delta.$

(ii) $(\Id_{L\otimes L\otimes L}+\xi+\xi^2)\circ(\Id_L\otimes \Delta)\circ\Delta=0.$
\end{Definition}

\begin{Proposition}\label{L5.2}
Let $(V,\Yup (z),c)$ be a graded vertex coalgebra. Set $\overline{V}=V/(\oplus_{i\neq1})V_i$, and define $\Delta:\overline{V}\rightarrow\overline{V}\otimes \overline{V}$ by $$\Delta(v)=\Res_z\Yup(z)v=\Delta_0(v).$$Then, $(\overline{V},\Delta)$ is an Lie coalgebra.
\end{Proposition}
\proof By skew-symmetry identity, for any $v\in \overline{V}, $we have
  \begin{eqnarray*}
    \tau\circ\Delta(v)&=&\Res_z\tau\circ\Yup(z)(v)=\Res_z\Yup(-z)\circ\exp^{zD^*}(v)\\
    &=&\sum_{k\geq0}(-1)^{-k-1}\frac{\Delta_k\circ (D^*)^k}{i!}(v)=-\Delta_0(v)=-\Delta(v).
  \end{eqnarray*}
The fourth identity follows since $D^*v=0$ in $\overline{V}$.

In Jacobi identity (\ref{JI3.4}), by taking $\Res_{z_0,z_1,z_2}$, we get$$(\Id_V\otimes \Delta_0)\circ\Delta_0-(\tau\otimes\Id_V)\circ(\Id_V\otimes\Delta_0)\circ\Delta_0=(\Delta_0\otimes\Id_V)\circ\Delta_0.$$Hence, on $\overline{V}$, we have
  \begin{eqnarray*}
    \xi\circ(\Id_{\overline{V}}\otimes \Delta_0)\circ\Delta_0&=&(\Id_{\overline{V}}\otimes \tau)\circ(\tau\otimes \Id_{\overline{V}})\circ(\Id_{\overline{V}}\otimes \Delta_0)\circ\Delta_0\\
    &=&(\Delta_0\otimes \Id_{\overline{V}})\circ\tau\circ\Delta_0=-(\Delta_0\otimes \Id_{\overline{V}})\circ\Delta_0,\\
    \xi^2\circ(\Id_{\overline{V}}\otimes \Delta_0)\circ\Delta_0&=&(\tau\otimes \Id_{\overline{V}})\circ(\Id_{\overline{V}}\otimes \tau)\circ(\Id_{\overline{V}}\otimes \Delta_0)\circ\Delta_0\\&=&-(\tau\otimes \Id_{\overline{V}})\circ(\Id_{\overline{V}}\otimes \Delta_0)\circ\Delta_0.
  \end{eqnarray*}
Thus $(\Id_{{\overline{V}}\otimes {\overline{V}}\otimes {\overline{V}}}+\xi+\xi^2)\circ(\Id_{\overline{V}}\otimes \Delta_0)\circ\Delta_0=(\Id_{{\overline{V}}\otimes {\overline{V}}\otimes {\overline{V}}}+\xi+\xi^2)\circ(\Id_{\overline{V}}\otimes \Delta)\circ\Delta=0$ on ${\overline{V}}$. This completes the proof.                                                     $\hfill\Box$

\end{document}